\documentclass[11pt]{amsart}

\usepackage{amsmath,amsthm,accents,amssymb, amscd, tikz, tikz-cd}
\usepackage{graphicx,enumerate,stmaryrd}
\usepackage{adjustbox}
\usepackage[all,cmtip]{xy}

\usetikzlibrary{decorations.markings}
\tikzstyle{vertex}=[circle, draw, inner sep=0pt, minimum size=6pt]

\usepackage[english]{babel}
\usepackage{comment}
\usepackage[margin=1in]{geometry}
\usepackage[bookmarks]{hyperref}
\usepackage{verbatim}

\usepackage[shortlabels]{enumitem}

\newtheorem{thm}{Theorem}[section]
\newtheorem{prop}[thm]{Proposition}

\newtheorem{lemma}[thm]{Lemma}
\newtheorem{cor}[thm]{Corollary}
		
\newtheorem{definition}[thm]{Definition}
\newtheorem{remark}{Remark}

\theoremstyle{definition}

\newcommand{\g}{{\mathfrak{g}}}

\newcommand{\del}{{\partial }}

\numberwithin{equation}{section}
\numberwithin{table}{section}

\begin{document}

\author{\sc Nikolay Grantcharov}
\address
{Department of Mathematics \\
University of Chicago   \\
Chicago \\ IL~60637}
\email{nikolayg@uchicago.edu}

\title{Infinitesimal Jet spaces of $\mathrm{Bun_G}$ in positive characteristic}
\begin{abstract}
Given a semisimple reductive group $G$ and a smooth projective curve $X$ over an algebraically closed field $k$ of arbitrary characteristic, let $\text{Bun}_G$ denote the moduli space of principal $G$-bundles over $X$. For a bundle $P\in\text{Bun}_G$ without infinitesimal symmetries, we provide a description of all divided-power infinitesimal jet spaces, $J_P^{n,PD}(\text{Bun}_G)$, of $\text{Bun}_G$ at $P$. The description is in terms of differential forms on $X^n$ with logarithmic singularities along the diagonals and with coefficients in $(\mathfrak{g}_P^*)^{\boxtimes n}$. Furthermore, we show the pullback of these differential forms to the Fulton-Macpherson compactification of the configuration space, $\hat{X}^n$, is an isomorphism. Thus, we relate the two constructions of \cite{BD,BG}, and as a consequence, give a connection between divided-power infinitesimal jet spaces of $\text{Bun}_G$ and the $\mathcal{L}ie$ operad.
\end{abstract}
\maketitle
\section{Introduction}
\subsection{} Let $X$ be a smooth projective curve over an algebraically closed field $k$ of characteristic $0$ or $p$. Let $G$ be a semi-simple reductive algebraic group over $k$. Let $\mathcal{M}\subset\text{Bun}_{G,X}(k)$ denote underlying smooth locus $\mathcal{M}$ of the moduli space of principal $G$-bundles over $X$, consisting of principal $G$-bundles with a finite group of automorphisms. Over an algebraically closed field of characteristic $0$, the problem of studying infinitesimal neighborhoods of a point in $\mathcal{M}$, i.e infinitesimal jet spaces, has been extensively studied. For example, there is a description in terms of deformation theory over the Ran space (\cite{Ran1,Ran2}), a description in terms hypercohomology of a complex built by ``higher Kodaira-Spencer classes'' (\cite{EV94}), and a description in terms of logarithmic differential forms on $X^n$ or $\hat{X}^n$, the Fulton-Macpherson compactification space (\cite{BD,BG,G}). 

The main theorems of the present paper are Theorem \ref{differential operators and forms on Sigma} and Theorem \ref{pullback of BD is BG}. Theorem \ref{differential operators and forms on Sigma}, and its analogue over $\hat{X}^n$, was stated in \cite{BD,BG,G} for characteristic 0 without proof. Our proof of this theorem, characteristic-independent, is motivated by \cite[Prop. 10.5.9]{BZF}, where the authors prove the analogues statement for the moduli space of line bundles on $X$. In fact, this proof may be adapted to other familiar moduli problems, e.g  moduli space of curves, or moduli space of $G$-bundles with flat connection, to show that divided-power infinitesimal jet spaces are identified with the corresponding sheaf of logarithmic differential forms. Next, Theorem \ref{pullback of BD is BG} shows that pulling back global sections of the \cite{BD} sheaf along the Fulton-Macpherson compactification $\hat{X}^n$ gives global sections of the \cite{BG} sheaf. This implies a relation between the divided-power infinitesimal jet spaces of $\text{Bun}_G$ and the $\mathcal{L}ie$-operad - Corollary \ref{jets and Lie operad}.

\noindent
\subsection{} The paper is organized as follows. In Section \ref{Residues}, we recall some general properties on differential forms with logarithmic singularities along a normal-crossing divisor. In Section \ref{Divided powers}, we recall the definition of a divided-power algebra and show divided-power infinitesimal jet spaces pair perfectly with the ``space of coinvariants'' of $\text{Bun}_G$. In Section \ref{universal sheaf}, we prove the main result that divided-power infinitesimal $n^{th}$-order jet space is isomorphic to a space of logarithmic differential forms on $X^n$ with special residue along the diagonals. This is summarized in diagram \ref{summary of jets}. Finally, in Section \ref{operads}, we prove the pull-back of global sections of \cite{BD} sheaf give global sections of \cite{BG} sheaf, and discuss how the latter is related to the $\mathcal{L}ie$ operad.

\subsection{Acknowledgements} The author expresses his sincerest gratitude to Victor Ginzburg for his constant guidance, support, and feedback throughout each step of this project. The author also thanks Sasha Beilinson for many fruitful conversations, especially for explaining \cite{BD} and sharing his proof of a key Lemma \ref{extending section from tubes to plane}. Finally, the author thanks Aaron Slipper for relentlessly answering the author's questions about $\text{Bun}_G$. The author was supported by NSF grant DGE 1746045.

\section{Residues on $X^n$}\label{Residues}
In this section, we recall some results on residues which will be used throughout this paper. As before, $X$ is a smooth projective curve over an arbitrary algebraically closed field $k$.
\subsection{Local Residue}\label{Local Residue}
Fix a closed point $x\in X$ on the curve and let $t_x$ be a uniformizer. There is an identification $ k[[t_x]]\cong O_x:=\hat{\mathcal{O}}_{X,x}$, the ring of regular functions on the completion of the stalk at $x$. And, on fraction fields we have $k((t_x))\cong K_x:= \text{frac}(\hat{\mathcal{O}}_{X,x})$. Let $D_x=\text{Spec}(\hat{\mathcal{O}}_{X,x}), D_x^\times = \text{Spec}(K(\hat{\mathcal{O}}_{X,x}))$ be the formal disk, formal punctured disk, respectively. 

The notion of a residue along a formal disk inside a smooth projective curve $X$ was first developed by \cite{Tate}. Given $\omega\in\Omega^1(X)$, express $\omega = f(t_x)dt_x$, for $f\in k((t_x))$, and write 
\begin{equation}\label{Tate residue}
\text{Res}_x\omega :=\text{Res}_{t_x=0}f(t_x)dt_x.
\end{equation}
Tate showed the residue is independent of choice of uniformizer. Furthermore, there is a ``strong residue theorem'':

\begin{thm}\label{SRT}\cite[Sect. 9.2.10]{BZF}(Strong Residue Theorem)
Let $\mathcal{E}$ be a locally free sheaf over $X$. Then a section $s\in\mathcal{E}(D_x^\times)$ extends to $\tilde{s}\in\mathcal{E}(X\setminus x)$ if and only if
$$\text{Res}_x(s,\omega)=0\;\;\text{ for all } \omega\in(\mathcal{E}^*\otimes\Omega_X)(X\setminus x).$$
As a corollary, there is a perfect pairing, given by the reside at $x$, between
$$\mathcal{E}(D_x^\times)/\mathcal{E}(X\setminus x)\times (\mathcal{E}^*\otimes\Omega_X)(X\setminus x)\xrightarrow{\text{Res}_x} k$$
\end{thm}

Next, we generalize Tate's residue to include formal neighborhoods inside $X^n$. This is used implicitly in \cite{BD}, but we make this computation explicit and algebraic. Let 
$$O_n=k[[z_1,\dots, z_n]],\;\;\;\; A_n=k[[z_1,\dots, z_n]][z_j^{-1}, (z_i-z_i)^{-1}]_{1\leq i<j\leq n}.$$

\noindent There is a natural inclusion $i_n:O_n\rightarrow A_{n-1}((z_{n-1}-z_n))$ given by $z_n\mapsto z_{n-1}-(z_{n-1}-z_n)$. By Taylor expanding, we see this inclusion extends uniquely to produce an algebra embedding $\Phi_n:A_n\rightarrow A_{n-1}((z_{n-1}-z_n))$. Similarly, there is a unique algebra embedding $\Phi_n:A_n\rightarrow A_{n-1}((z_n))$ extending the natural inclusion $O_n\rightarrow A_{n-1}((z_n))$

\begin{definition}\label{formal residue}
Suppose $f\in A_n$. We say the expansion of $f$ in the $z_{n-1}-z_n$ direction is 
$$\Phi_{n-1,n}(f)=\sum_{i\geq d}P_i\cdot(z_{n-1}-z_n)^i\in A_{n-1}((z_{n-1}-z_n)).$$
We define the residue map 
\begin{align*}
\text{Res}_{z_{n-1}=z_n}(f) &:=\text{ The }(z_{n-1}-z_n)^{-1} \text{ coefficient of } \Phi_{n-1,n}(f).\\
\text{Res}_{z_n=0}(f)&:=\text{ The }z_n^{-1}\text{ coefficient of } \Phi_n(f).
\end{align*}
In both cases, the residue map is an $O_n$-linear map $A_n\rightarrow A_{n-1}$.
\end{definition}
Note, we could have inverted any irreducible polynomials in $O_n$, and defined the ``local residue'' of $A_n$ along these polynomials appropriately. Finally, we recall a fundamental result on residues:
\begin{lemma}(Parshin)\label{Parshin}
Suppose $f\in A_3$. Then 
$$\text{Res}_{z_1=0}\text{Res}_{z_2=0}(f)-\text{Res}_{z_2=0}\text{Res}_{z_1=0}(f) = -\text{Res}_{z_1=0}\text{Res}_{z_1=z_2}(f)$$
\end{lemma}
\begin{proof}
This can be very easily checked by brute force on $f=\frac{1}{z_1^az_2^b(z_1-z_2)^c}$ for $a,b,c\in\mathbf{Z}_{\geq0}$. \end{proof}

\subsection{Global Residue}
We recall the definition of logarithmic differential forms along a reduced, normal crossing divisor $D$ of a smooth algebraic variety $S$ over $k$. We also construct the residue map and recall several of its basic properties. As $k$ is allowed to be an arbitrary field, we work over the algebraic (etale) topology. A nice survey of these classical facts and constructions may be found in the original source, Deligne Hodge II, or e.g \cite{EV92, S79}.

Since $D$ is normal crossing in $S$, we may cover $S$ via open affine $U$ such that 
\begin{itemize}
\item $U$ is etale over $\mathbf{A}_k^n$, via coordinates $x_1,\dots, x_n$.
\item $D\vert U$ is defined by an equation $x_1\dots x_s=0$ (i.e $D$ is the inverse image of the union of the first $s\leq n$ coordinate hyperplanes of $\mathbf{A}_k^n$).
\end{itemize}

\noindent Define $\Omega_S^\bullet(\text{log}D)$, the locally-free $\mathcal{O}_S$-module of logarithmic differential forms on $S$ along $D$, as the subsheaf of $j_*(\Omega_{S\setminus D}^\bullet)$, where $j:S\setminus D\hookrightarrow S$ is the inclusion, consisting (locally over $D\vert U=(h=0)$) of forms $\omega$ such that $h\omega,dh\wedge\omega\in\Omega_S^\bullet(U)$. Equivalently, we may specify that over each open affine $U$ as above, $\Omega_S^1(\text{log}D)(U)$ has basis $\frac{dx_1}{x_1},\dots,\frac{dx_s}{x_s},dx_{s+1},\dots, dx_n$, and define $\Omega^i_{S}(\text{log}D)=\wedge^i_{\mathcal{O}_S}(\Omega_S^1(\text{log}D))$.

Next we explain how to compute the residue of a logarithmic form. Suppose $p\in D$, and choose local coordinates $x_1,\dots, x_n$ of $D\vert U=(h_p=0)$ around $p$. Since $D$ is normal crossing, we may assume $h_p=x_1\cdots x_{s'}$ for $s'\leq s\leq n$, and that $(D,p) = (D_1,p)\cup\dots\cup (D_{s'},p)$ where $(D_i,p)=(x_i=0)$ are the irreducible components of the divisor $D$ locally around $p$. Using the above notation, suppose $\omega_p\in\Omega^1_{S,p}(\text{log}D)$, and consider a local expansion 
$$\omega_p = \sum_{i=1}^{s'}f_i\frac{dx_i}{x_i}+\sum_{i=s'+1}^n g_idx_i$$
where $f_i,g_i\in\mathcal{O}_{S,p}$. Define $\text{Res}_{D,p}$ to be the $\mathcal{O}_{S,p}$-linear map
\begin{equation}\label{Residue along NCD}
\text{Res}_{D,p}:\omega_p\mapsto (f_1\vert_{D_1,p},\dots, f_{s'}\vert_{D_{s'},p})\in\bigoplus_{i=1}^{s'}\mathcal{O}_{D_i,p}.
\end{equation}

That the residue is independent of choice of presentation follows from the ``generalized de Rham lemma'', which, as formulated in \cite{S76}, is characteristic independent. Further details may be found in e,g \cite{S79}. There is an exact sequence of coherent $\mathcal{O}_S$-modules:
\begin{equation}\label{residue short exact sequence}
0\rightarrow \Omega_S^1\rightarrow\Omega_S^1(\text{log}D)\xrightarrow{\text{Res}_D}\bigoplus_{j=1}^s(i_j)_*\mathcal{O}_{D_j}\rightarrow0
\end{equation}
where $i_j:D_j\hookrightarrow S$ is the inclusion. For differential forms of higher order, there is again a residue:
$$\Omega_S^p(\text{log}D)\xrightarrow{\text{Res}_D}\bigoplus_{j=1}^s(i_j)_*\Omega_{D_j}^{p-1}(\text{log}(D- D_j))$$
whose construction we now explain. Again, $\text{Res}_D$ is defined as the sum of residues $\text{Res}_{D_i}$ along each irreducible component $D_i$ of $D$. Working locally over the component $D_1=(x_1=0)$, we may write $\omega\in\Omega_S^p(\text{log}D)$ as $\omega=\frac{dx_1}{x_1}\wedge\xi+\eta$ where $\xi\in\Omega_S^{p-1}(\text{log}(D-D_1))$ and $\eta\in\Omega_S^p(\text{log}(D-D_1))$. Then we define 
$$\text{Res}_{D_1}(\frac{dx_1}{x_1}\wedge\xi+\eta):=\xi\vert_{D_1}\in\Omega_{D_1}^{p-1}(\text{log}(D-D_1)\vert_{D_1}).$$
This fits into an exact sequence
\begin{equation}\label{General residue exact sequence}
0\rightarrow\Omega_S^p(\text{log}(D-D_1))\rightarrow\Omega_S^p(\text{log}D)\xrightarrow{\text{Res}_{D_1}}(i_1)_*\Omega_{D_1}^{p-1}(\text{log}(D-D_1)\vert_{D_1})\rightarrow0.
\end{equation}
\noindent Finally, we cite a useful lemma whose proof may be found in \cite[Lemma 1.6]{V82}.

\begin{lemma}\label{pullback of NCD}
Let $f:X\rightarrow Y$ be a smooth map of algebraic varieties and suppose $D$ is a normal crossing divisor in $Y$. Suppose $\hat{D}:=(f^{-1}(D))_{\text{red}}$ is a normal-crossing divisor of $X$. Then $f^*(\Omega_Y(\text{log}(D)))\subset\Omega_X(\text{log}(\hat{D}))$
\end{lemma}

\section{Jets on $\text{Bun}_G$}\label{Divided powers}

\subsection{Divided power algebras}
We recall the notion of divided power structure, following the exposition in \cite[Section 3]{BO}. Let $A$ be a commutative ring and $I$ an ideal. Then a \textit{divided power structure on $I$} means a collection of maps $\gamma_i:I\rightarrow A$, for all integers $i\geq 0$, satisfying: \begin{itemize}
\item For all $x\in I$, $\gamma_0(x)=1,\gamma_1(x)=x, \gamma_i(x)\in I$ if $i\geq 1$.
\item For $x,y\in I$, $\gamma_k(x+y)=\sum_{i+j=k}\gamma_i(x)\gamma_j(y)$.
\item For $\lambda\in A,x\in I, \gamma_k(\lambda z)=\lambda^k\gamma_k(x)$.
\item For $x\in I,\gamma_i(x)\gamma_j(x)=\binom{i+j}{i}\gamma_{i+j}(x)$.
\item $\gamma_p(\gamma_q(x))=\frac{(pq)!}{p!(q!)^p}\gamma_{pq}(x)$.
\end{itemize}
We will also say \textit{$(A,I,\gamma)$ is a P.D ring}. The motivation for divided power structure is to replace elements $\frac{x^n}{n!}$, which do not make sense in characteristic $p$, with elements $\gamma_n(x)$. The 5 axioms then capture all the properties we would like for $\gamma_n(x)$ to satisfy, were it to be $\frac{x^n}{n!}$. It is a theorem (\cite[Theorem 3.9]{BO}) that given any commutative $k$-algebra $A$ and an $A$-module $I$, there exists a P.D algebra, denoted $(\Gamma_A(I),\Gamma_A^+(I),\gamma)$, satisfying a natural universal property. We call $\Gamma_A(I)$ the divided power envelope of $I$ in $A$. There is a grading on $\Gamma_A(I)$ with $\Gamma_0(I)=A,\Gamma_1(I)=I,$ and $\Gamma_A^+(I)=\oplus_{i\geq 1}\Gamma_i(I)$. In particular, there is an $A$-linear map $I\rightarrow\Gamma_A(I)^+$.

Let us describe the \textit{P.D polynomial algebra.} Suppose $\{x_1,\cdots, x_n\}$ is an $A$-module basis for $I$. Denote $x^{[n]}=\gamma_n(x)\in\Gamma_n(I)$. Then
$$\{x^{[q]}:=x_1^{[q_1]}\dots x_n^{[q_n]}:\sum q_i=n, x_i\in I\}\;\;\text{ forms an A-module basis for }\Gamma_n(I)$$
and we denote $\Gamma_A(I):=A\langle x_1,\dots x_n\rangle$ and call it the P.D polynomial $A$-algebra. For example, let $V=k^n$ be a vector space and consider $A:=\text{Sym}(V):=TV/\langle a\otimes b-b\otimes a\rangle$, and $I=\text{Sym}(V)^+$. Then 
$$\Gamma_A(I)=k\langle x_1,\dots x_n\rangle \simeq \text{Sym}^{PD}(V):= \bigoplus_{i\geq 0} \big( V^{\otimes n}\big)^{S_n},$$
where the isomorphism is given by $x^{[q]}\mapsto \sum_{\omega\in O_q}\omega$ and $O_q$ is the orbit $S_{q_1+\dots+q_n}.x_1^{\otimes q_1}\otimes\dots\otimes x_n^{\otimes q_n}$.
One readily checks the natural pairing of $V,V^*$ induces the perfect pairing
$$\langle,\rangle:\; \text{Sym}^{PD}(V)\times \text{Sym}(V^*)\rightarrow k.$$

\subsection{Divided power jet spaces}

In this subsection, we recall the definition of divided power jet spaces, as done in \cite{BO}. Suppose $\mathcal{M}$ is a smooth scheme over $k$, and denote $\mathcal{O}:=\mathcal{O}_{\mathcal{M}},\; \mathcal{O}^e:=\mathcal{O}_{\mathcal{M}}\otimes\mathcal{O}_{\mathcal{M}}.$ Let $\Delta:\mathcal{M}\hookrightarrow \mathcal{M}\times \mathcal{M}$ be the diagonal embedding and $I$ the ideal sheaf of $\Delta$ in $\mathcal{M}\times\mathcal{M}$. Then $I=\mathcal{O}\langle a\otimes 1-1\otimes a\rangle$ is the kernel of the multiplication map $\mathcal{O}^e\rightarrow\mathcal{O}$. The usual \textit{sheaf of $n$-jets of functions on $\mathcal{M}$} (without divided-powers) is defined as  
\begin{equation}\label{usual jets}
J^n(\mathcal{M}):=\mathcal{O}^e/I^{n+1}.
\end{equation}
This is naturally a $\mathcal{O}^e$-module. The fiber over $P\in\mathcal{M}$ will be denoted $J^n_P(\mathcal{M})$ and is called the vector space of \textit{$n^{th}$-order infinitesimal jet-spaces of $P\in\mathcal{M}$}.
\noindent For example, $J^1(\mathcal{M})=\mathcal{O}_{\mathcal{M}}\oplus T_{\mathcal{M}}^*$ and $J^1_P(\mathcal{M})\simeq k\oplus T_P^*(\mathcal{M})$ is (scalars plus) the Zariski cotangent space. Considering $J^n(\mathcal{M})$ as a left $\mathcal{O}$-module, we define the  sheaf of \textit{Grothendieck differential operators of order at most $n$ on $\mathcal{M}$ } as
\begin{equation}\label{Grothendieck diffops}
\mathcal{D}^{Gr}(\mathcal{M})_n:=\mathcal{H}om_{\mathcal{O}}(J^n(\mathcal{M}),\mathcal{O}).
\end{equation}

Next, we explain the divided-power analogues. The divided power envelope of $I$ inside $\mathcal{O}^e$ is (etale) locally isomorphic to the P.D polynomial algebra:
\begin{equation} \label{smooth immersion divided powers}
\Gamma_{\mathcal{O}^e}(I)\simeq\mathcal{O}\langle \xi_1,\dots,\xi_n\rangle,
\end{equation}
 where $\xi_i=1\otimes x_i-x_i\otimes 1$ are the $\mathcal{O}$-basis of $I$, and $x_i$ are the local coordinates of $\mathcal{O}$. Let $(\Gamma_{\mathcal{O}^e}(I),\bar{I},\gamma)$ be the associated P.D. ring. Define the \textit{ $n^{th}$-order divided power neighborhood}, also called the sheaf of $n$-PD jets, by 
\begin{equation}\label{PD jets}
J^{n,PD}(\mathcal{M}):=\Gamma_{\mathcal{O}^e}(I)/\bar{I}^{n+1}.
\end{equation}
The fiber over $P\in\mathcal{M}$ is denoted $J^{n,PD}_P(\mathcal{M})$ and called the \textit{$n^{th}$-order infinitesimal divided-power jet-spaces of $P\in\mathcal{M}$}. Viewing $J^{n,PD}_P(\mathcal{M})$ as a left $\mathcal{O}$-module, we define the sheaf of P.D differential operators of order at most $n$, also called the sheaf of \textit{crystalline differential operators}, by
$$\mathcal{D}^{crys}(\mathcal{M})_n:=\mathcal{H}om_{\mathcal{O}}(J^{n,PD}(\mathcal{M}),\mathcal{O}).$$
\noindent By definition,
$$\langle,\rangle: \mathcal{D}^{crys}_P(\mathcal{M})_n\times J^{n,PD}_P(\mathcal{M})\rightarrow k\;\;\text{ is a perfect pairing.}$$
We caution that the pairing is \textit{not} given by evaluating a crystalline differential operator on a function with divided powers!

Recall in Equation \ref{smooth immersion divided powers} the $\mathcal{O}$- basis $\{\xi^{[q]}:=\xi_1^{[q_1]}\dots \xi_n^{[q_k]}:|q|\leq m\}$ of $J^{m,PD}(\mathcal{M})$. Let $D_q$ denote the corresponding dual basis of $\mathcal{D}^{crys}(\mathcal{M})$. The composition in $\mathcal{D}^{crys}(\mathcal{M})$ is defined as follows. Let
$$\delta:\mathcal{O}^e\rightarrow\mathcal{O}^e\otimes_{\mathcal{O}}\mathcal{O}^e,\;\; \xi\mapsto \xi\otimes\mathbf{1}+\mathbf{1}\otimes\xi\;\text{ if }\xi=1\otimes x-x\otimes 1, \mathbf{1}=1\otimes 1.$$
By universal property, this induces a map $\delta:\Gamma_{\mathcal{O}^e}(I)\rightarrow \Gamma_{\mathcal{O}^e}(I)\otimes_{\mathcal{O}} \Gamma_{\mathcal{O}^e}(I)$ which satisfies
$$\delta^{n,m}:J^{n+m,PD}(\mathcal{M})\rightarrow J^{n,PD}(\mathcal{M})\otimes_{\mathcal{O}} J^{m,PD}(\mathcal{M}),\;\; \xi^{[q]}\mapsto \sum_{i+j=q}\xi^{[i]}\otimes\xi^{[j]}.$$
\noindent Then define composition of $f\in\mathcal{D}^{crys}(\mathcal{M})_{n},g\in\mathcal{D}^{crys}(\mathcal{M})_{m}$ is defined by the formula
$$f\circ g: J^{n+m,PD}(\mathcal{M})\xrightarrow{\delta^{n,m}}J^{n,PD}(\mathcal{M})\otimes_{\mathcal{O}} J^{m,PD}(\mathcal{M})\xrightarrow{1\otimes f}J^{n,PD}(\mathcal{M})\xrightarrow{g} \mathcal{O}$$
It is immediate from this formula that we have the following relations inside $\mathcal{D}^{crys}(\mathcal{M})$:
\begin{equation}\label{crystalline relations}
D_{q}\circ D_{q'}=D_{q+q'}\;\;\;\text{ and } fD_q = \sum_{|i|+|j|=|q|}\binom{|i|+|j|}{|i|}D_i(f)D_j \text{ for }f\in J^{PD}(\mathcal{M})
\end{equation}

The \textit{universal enveloping algebra of the tangent Lie algebroid $\mathcal{T}_X$}, denoted $U_{\mathcal{O}_X}(\mathcal{T}_X)$, is defined as the $\mathcal{O}_X$ algebra generated by $\mathcal{O}_X$ and $\mathcal{T}_X$, subject to the relations 
$f\cdot \del= f\del, \del\cdot f-f\cdot \del = \del(f), \del\cdot\del'-\del'\cdot\del = [\del,\del'],\;\;\text{ for } f\in \mathcal{O}_X,\del,\del'\in\mathcal{T}_X.$ By Equation \ref{crystalline relations}, it is clear
\begin{equation}\label{crystalline = enveloping}
\mathcal{D}^{crys}(\mathcal{M})\xrightarrow{\sim} U_{\mathcal{O}_{\mathcal{M}}}(\mathcal{T}_{\mathcal{M}}),\;\; D_q\mapsto D_1^{q_1}\cdot D_2^{q_2}\cdots D_n^{q_n}
\end{equation}
is an isomorphism of $\mathcal{O}_{\mathcal{M}}$-algebras (see also \cite{BMR}).

\subsection{Pairing fibers of jets with the coinvariants}
As a starting point, we cite the uniformization theorem for $\text{Bun}_G$. Recall $O_x\simeq k[[t_x]], K_x\simeq k((t_x))$, and denote $\mathcal{O}_{out}:=\mathcal{O}_X\vert_{X\setminus x}$.
\begin{thm}(Uniformization Theorem)
Let $G$ be a semisimple reductive group over $k$, $X$ a smooth projective curve over $k$, and $x\in X$ a closed point. Then there is an isomorphism of stacks 
$$\text{Bun}_{G}(X)=G(\mathcal{O}_{out})\backslash G(K_x)/G(O_x).$$
\end{thm}
\noindent A proof of the uniformization theorem, as formulated, may be found in \cite[Theorem 5.1.1]{Sor99}. We will use the uniformization theorem to produce an explicit description of the infinitesimal jet spaces on the underlying smooth locus $\mathcal{M}$ of $\text{Bun}_G$. In particular, we only consider $P\in\text{Bun}_G$ whose associated adjoint bundle $\mathfrak{g}_P$ has no global sections: $H^0(X,\mathfrak{g}_P)=0$.

Let 
$$\mathfrak{g}_K:=\mathfrak{g}\otimes K,\mathfrak{g}_O:=\mathfrak{g}\otimes O.$$
Under the uniformization map, we may associate to $P$ the triple $(\tau_{X\setminus x},\tau_{D_x},\phi)$, where 
$$\tau_{D_x}: P\vert_{D_x}\xrightarrow{\sim}G\times D_x,\;\;\; \tau_{X\setminus x}:P\vert_{X\setminus x}\xrightarrow{\sim}G\times(X\setminus x)$$
are trivializations and $\phi\in G(K)$ is the transition function on the overlap $D_{x}^\times$. This means the composition of maps
\begin{equation}\label{trivialization}
\mathfrak{g}_{\text{out}}   \xrightarrow{\tau_{X\setminus x}^{-1}} \Gamma(X\setminus x,\mathfrak{g}_P)\xrightarrow{\text{Restrict}}\Gamma(D_x^\times,\mathfrak{g}_P)\xrightarrow{\tau_{D_x}}\mathfrak{g}_K
\end{equation}
has image landing in $\text{Ad}_\phi\mathfrak{g}_O\subset\mathfrak{g}_K$. 

\begin{definition}\label{vacuum module} Let $\phi\in G(K)$. The ``vacuum module with central charge 0'' is
$$\mathbf{M}_{\phi}:= U(\mathfrak{g}_K)/{\text{Ad}_{\phi}\mathfrak{g}_O\cdot U(\mathfrak{g}_K)}.$$
The ``space of coinvariants'' is
$$\mathbf{M}^{out}_{\phi}:=\mathbf{M}/\mathbf{M}\mathfrak{g}_{out},$$
 where $\mathbf{M}_\phi$ is viewed as a $\mathfrak{g}_{out}$-module under the embedding $\mathfrak{g}_{out}\hookrightarrow\mathfrak{g}_K$ in Equation \ref{trivialization}.
 \end{definition}
The PBW filtration on $U(\mathfrak{g}_K)$ induces a PBW filtration on $\mathbf{M}_{\phi}^{out}$, where the nth order filtration is denoted $\mathbf{M}_{\phi,n}^{out}$.  
 Explicitly, we find
\begin{align*}
\mathbf{M}_{\phi,n}&:=\text{span}\{\xi_1\dots\xi_r:r\leq n, \xi_i\in\mathfrak{g}_K\}/\text{span}\{\xi_1\dots\xi_r:r\leq n,\xi_i\in\mathfrak{g}_K,\xi_1\in\phi\mathfrak{g}_O\phi^{-1}\}\\
\mathbf{M}_{\phi,n}^{out}&:=\text{span}\{\xi_1\dots\xi_r:r\leq n, \xi_i\in\mathfrak{g}\}/\text{span}\{\xi_1\dots\xi_r:r\leq n,\xi_i\in\mathfrak{g}_K,\xi_1\in\phi\mathfrak{g}_O\phi^{-1}\;\text{ or }\; \xi_r\in\mathfrak{g}_{out}\}
\end{align*}

\begin{lemma}
There is an isomorphism of vector spaces between the fiber of crystalline differential operators on $\mathcal{M}$ and the space of coinvariants:
$$\mathcal{D}^{crys}_P(\mathcal{M})\simeq \mathbf{M}^{out}$$
\end{lemma}
\begin{proof}
By \ref{crystalline = enveloping}, $\mathcal{D}^{crys}(\mathcal{M})\simeq U_{\mathcal{O}_M}(\mathcal{T}_\mathcal{M})$. Thus the fiber over $P\in\mathcal{M}$ is just $U(T_P(\mathcal{M}))$, the enveloping algebra of the tangent space $T_P(\mathcal{M})$. It is well known, e.g by using the uniformization theorem for $\text{Bun}_G$, that $T_P\mathcal{M}$ is isomorphic to $\mathfrak{g}_K\big/(\text{Ad}_\phi\mathfrak{g}_O+\mathfrak{g}_{\text{out}}\big)$. The claim $U(T_P(\mathcal{M}))\simeq\mathbf{M}^{out}$ follows by comparing the associated graded of both sides.
\end{proof}
  \begin{cor}\label{jets and differential operators}
The natural pairing of crystalline differential operators with divided-power jets induces a perfect pairing
 $$\phi:\mathbf{M}_{\phi,n}^{out}\times J^{n,PD}_P(\mathcal{M})\rightarrow k.$$
 \end{cor}

\section{The ``universal sheaf'' of log-differential forms}\label{universal sheaf}
Fix a closed point $x\in X$. Following the notation of subsection \ref{Local Residue}, we consider the $n$-fold formal, resp. punctured disk on $X^n$:
$$(D_x)^n:=\text{Spec}((\widehat{\mathcal{O}}_{X,x})^{\hat{\otimes} n}), \;\;\;\;\; (D_x^\times)^n:=\text{Spec}(K(\widehat{\mathcal{O}}_{X,x}^{\hat{\otimes} n})).$$
So, after choosing uniformizers, we may identify
$$\Gamma((D_x)^n,\mathcal{O}_{X^n})=k[[z_1,\dots, z_n]],\;\;\; \Gamma((D_x^\times)^n,\mathcal{O}_{X^n})=\text{Frac}(k[[z_1,\dots,z_n]]).$$
In higher variables, the trivialization (\ref{trivialization}) reads
$$\tau_{D_x}:\Gamma(D_x^n,\mathfrak{g}_P^{\boxtimes n})\xrightarrow{\sim} \mathfrak{g}^{\otimes n}\otimes k[[z_1,\dots, z_n]] = (\mathfrak{g}_O)^{\hat{\otimes} n}.$$\;\;\;  
Let $S_n$ denote the symmetric group of $n$ letters. For $\sigma\in S_n$, let
$$D_\sigma=\bigcup_{i=1}^{n-1}\{x_{\sigma(i)}=x_{\sigma(i+1)}\}\subset X^n.$$
 Observe each $D_\sigma$ is a normal crossing divisor, but the full diagonal divisor $D:=\bigcup_{\sigma\in S_n}D_\sigma$ is not as soon as $n>2$. Let $\Omega_{X^n}(\text{log}(D_\sigma))$ denote the sheaf of top degree log-differential forms on $X^n$ with simple poles along $D_\sigma$, as defined in Section 2. From here on, we always mean top-degree forms on a smooth scheme $S$ when we write $\Omega_S$. 

\begin{definition}\label{log forms on almost NCD} Let $j:X^n\setminus D\hookrightarrow X^n$ be the inclusion. Define \footnote{To avoid overcount, we can index the sum over $S_n/\tau$, where $\tau:i\mapsto n-i-1$ is the Type $A_n$ Dynkin automorphism.}
$$\tilde{\Omega}_{X^n}(\text{log}D):=\sum_{\sigma\in S_n}\Omega_{X^n}(\text{log}D_\sigma)\subset j_*(\Omega_{X^n\setminus D})$$
\end{definition}

\noindent Thus, restriction to $(D_x^{\times})^n$ and then trivializing about $\tau_{D_x}$ produces a map
$$\Gamma(X^n,(\mathfrak{g}_P^*)^{\boxtimes n}\otimes\tilde{\Omega}_{X^n}(\text{log}D))\hookrightarrow (\mathfrak{g}^*)^{\otimes n}\otimes k[[z_1,\dots,z_n]][(z_j-z_k)^{-1}]dz_1\cdots dz_n,\;\;\; \omega_n\mapsto\omega_n\vert_{(D_x^{\times})^n}.$$
\noindent By Equation \ref{trivialization}, the image lands in
\begin{equation}\label{image of trivialization}\omega_n\vert_{(D_x^\times)^n}\in(\text{Ad}_\phi\mathfrak{g}^*_O)^{\hat{\otimes} n}[(z_i-z_j)^{-1}, i,j\leq n]dz_1\cdots dz_n.\end{equation}
Similarly, for $\omega_n\in\Gamma(D_x^n,(\mathfrak{g}_P^*)^{\boxtimes n}\otimes\tilde{\Omega}_{X^n}(\text{log}D))$, we denote by $\omega_n\vert_{(D_x^\times)^n}$ to be the restriction to the formal punctured disk, followed by the $\tau_{D_x}$-trivialization. This still has image as in Equation \ref{image of trivialization}.

So in what follows, we study properties of log-differential forms on $X^n$ by studying their corresponding power-series expansion. 
We are now ready to define the main object:
\subsection{The universal space controlling jets}
\begin{definition}\label{The universal space} Define the ``universal space'' $\Omega_n$ over $X$ as follows. Over an open $U\subset X$, define the sections of $\Omega_n(U)$ to consist of $(n+1)-$ tuples
$$(\omega_0,\dots,\omega_n)$$
where $\omega_0\in k$ and $\omega_i\in\Gamma(U^i,(\mathfrak{g}_P^*)^{\boxtimes i}\otimes \tilde{\Omega}_{X^i}(\text{log}D))^{-S_i}, i>0$ are such that, after
restricting to $(D_x)^i$, 
$$\omega_i\vert_{(D_x)^i}\in \bigg((\mathfrak{g}^*)^{\otimes i}\otimes k[[z_1,\dots, z_i]][(z_j-z_k)^{-1}]\bigg)^{-S_i}dz_1\dots dz_i$$
and $\omega_i\vert_{(D_x)^i}$ has the expansion in the $z_{i-1}-z_i$ direction (as defined in Definition \ref{formal residue}):
\begin{equation}\label{residue constraint}
\omega_i\vert_{(D_x)^i}=\frac{\Phi_i^*(\omega_{i-1}\vert_{(D_x)^{i-1}})}{z_{i-1}-z_i}dz_i+\text{reg}.
\end{equation}
Here, $(-)^{-S_i}$ denotes the $S_i$-anti-invariants and $\Phi_i^*$ is the map induced by the Lie-bracket:
$$\Phi_i: \mathfrak{g}^{\otimes i} \rightarrow \mathfrak{g}^{\otimes i-1},\;\;\; \xi_1\otimes\cdots\otimes \xi_i\mapsto\xi_1\otimes\cdots\otimes [\xi_{i-1},\xi_{i}].$$
And, ``reg'' stands for an expression whose expansion in the $z_{i-1}-z_i$ direction has no $(z_{i-1}-z_i)^{d}$ terms when $d<0$.
\end{definition}

\begin{remark}\label{justify NCD log forms}
Let us explain the reason for using $\tilde{\Omega}_{X^n}(\text{log}D)$ in defining the space of sections $\Omega_n(U)$ instead of the more natural $\Omega_{X^n}(\text{log}D)$.

Following \cite{S79}, it is possible to generalize the notion of logarithmic differential forms and their residue along divisors which are not normal crossing. The main difference is the residue in Equation \ref{residue short exact sequence} lands in rational sections of the (normalization of the) divisor instead a regular sections. In our case, the residue constraint \ref{residue constraint} ensure this does not happen:

Indeed, suppose $\omega=(\omega_i)_{0\leq i\leq n}\in\Omega_n(U)$ satisfies all the conditions of Definition \ref{The universal space} except with $\Omega_{X^n}(\text{log}D)$ replacing $\tilde{\Omega}_{X^n}(\text{log}D)$. Suppose there was a term of $\omega_n\vert_{((D_x)^{\times})^n}$ with a ``loop'' in the denominator: $z_{i_1i_2}z_{i_2i_3}\dots z_{i_{m}i_1}$. Then the iterated residue
$$\tilde{\omega}_{n-m+2}:=\text{Res}_{z_{i_1i_{2}}}\circ\text{Res}_{z_{i_{2}i_{3}}}\circ\dots\circ\text{Res}_{z_{i_{m-2},i_{m-1}}}(\omega_n)$$
creates an order 2-pole for $\tilde{\omega}_{n-m+2}$ at $z_{i_1i_{m}}$.  For example, if $m=3$, $\text{Res}_{z_1=z_2}\frac{1}{z_{12}z_{23}z_{31}} = \frac{-1}{z_{13}^2}.$ 

But the composition of Lie bracket maps, $\Phi_T^*$ (see \ref{Phi_T^* plain}), is $k((z_1,\dots,z_n))$-linear. Hence\linebreak $(\Phi_{T}^*)^{-1}(\tilde{\omega}_{n-m+2})=\omega_{n-m+2}$ should still have log poles on the diagonals. Contradiction. Thus, there exist no ``loops'' in $\omega_n$, and we conclude if $z_{i_1i_2}\dots z_{i_{m-1}i_m}$ is a denominator of $\omega_n\vert_{(D_x^{\times})^n}$, then $m\leq n$ and $i_j$ are all distinct. So the two options produce the same definition for $\Omega_n(U)$.
\end{remark}

We will only consider  the ``universal space'' $\Omega_n$ on $U=D_x^\times, D_x, X$. In these cases, we prove $\Omega_n(U)$ is naturally isomorphic to the infinitesimal jet spaces of $G(K), \text{Gr}_G, \text{Bun}_G$, respectively. As defined, the sheaf $\Omega_n(X)$ depends on $x\in X$ since we check the residue condition (Equation \ref{residue constraint}) locally around the formal disk about $x$. As a consequence of Theorem \ref{differential operators and forms on Sigma} or Subsection \ref{2 point uniformization}, this definition is in fact independent of the choice of point on the curve.

Let us now formulate the main results of \cite{BD} as Proposition \ref{differential operators and forms on K} and Proposition \ref{differential operators and forms on K/O}.

\begin{prop}\label{differential operators and forms on K}\cite{BD}
Let $k$ be an arbitrary field. Then there is a perfect pairing 
$$\Phi_K: U^{\leq n}(\mathfrak{g}_K)\times \Omega_n(D_x^{\times})\rightarrow k$$
induced by the pairing: given $\xi_1,\dots,\xi_k\in \mathfrak{g}_K$ and $\omega:=(\omega_0,\dots,\omega_n)$, 
\begin{equation}\label{residue pairing on K}\phi_K(\xi_1\dots\xi_k,\omega):= \text{Res}_{z_1=0}\dots\text{Res}_{z_k=0}\langle \xi_1(z_1)\otimes\dots\otimes\xi_k(z_k),\omega_k\rangle
\end{equation}
where $ \xi_1(z_1)\otimes\dots\otimes\xi_k(z_k)\in(\mathfrak{g}_K)^{\hat{\otimes} i}$ and $\omega_k$ is a differential form with values in $(\mathfrak{g}^*)^{\otimes k}.$ So, $\langle,\rangle$ is a scalar-valued differential form. And, $\text{Res}_{z_1=0}\dots\text{Res}_{z_k=0}$ means we first compute the residue $z_k=0$ and treat $z_1,\dots, z_{k-1}$ as scalars, etc, as explained in Definition  \ref{formal residue}.
\end{prop}
\begin{remark} A heuristic explanation for why to expect \cite{BD} to hold in characteristic $p$ is that under the perfect pairing, $\Omega(D_x^\times)$ plays the role of the algebra with divided powers and $U(\mathfrak{g}_K)$ plays the role of the algebra without divided powers. A more convincing, general construction of a ``divided power module'' will be considered in Subsection \ref{Jet spaces and the Lie cooperad}.
\end{remark}
\begin{proof}
For the reader's convenience, we highlight the key points in the argument. First, the residue pairing on $\mathfrak{g}_K^{\hat\otimes n}\times\Omega_n(D_x^\times)\rightarrow k$ descends to $U^{\leq n}(\mathfrak{g}_K)$ because of Parshin's residue formula \ref{Parshin}.

We will show the residue pairing is perfect by constructing an isomorphism between $(U^{\leq n}\mathfrak{g}_K)^*$ and  $\Omega_n(D_x^\times)$. Let $e_i\in\mathfrak{g}$ be a basis and $e^i$ the dual basis of $\mathfrak{g}^*$. Let $e_i^{(l)}=e_iz^l\in\mathfrak{g}_K\subset U(\mathfrak{g}_K)$.
Define the map 
$$\psi_K:(U^{\leq n}\mathfrak{g}_K)^*\rightarrow \Omega_n(D_x^\times),\;\;\text{ by }$$ 
\begin{equation}\label{expansion of poles}\lambda\mapsto (w_r),\;\;\text{ where } w_r=\sum_{l_1,\dots, l_n\in\mathbf{Z}}\sum_{i_1,\dots, i_r}\lambda(e_{i_1}^{(l_1)}\dots e_{i_r}^{(l_n)})e^{i_1}\otimes\dots\otimes e^{i_r} z_1^{-l_1-1}\dots z_r^{-l_r-1}dz_1\dots dz_r
\end{equation}
As written, we have 
$$\omega_r\in (\mathfrak{g}^*)^{\otimes r}\otimes k((z_1))\cdots ((z_r))dz_1\dots dz_r.$$
To see that $(\omega_r)$ lands in $\Omega_n(D_x^\times)$, we first show $\omega_r$ is in fact the expansion of an element 
$\hat{\omega}_r\in \big((\mathfrak{g}^*)^{\otimes r}\otimes k[[z_1,\dots, z_r]][(z_i-z_j)^{-1}z_k^{-1}]dz_1\dots dz_r\big)^{S_r}$, where by expansion we mean as introduced  in Definition \ref{formal residue}. We make the subtle remark that $S_r$ does not act on the ``expansion'' because $k((z_1))((z_2))\neq k((z_2))((z_1))$. So we actually must define the map to be $\psi_K:\lambda\mapsto (\hat{\omega}_r)$, once we show $\omega_r$ is the expansion of some (unique) $\hat{\omega}_r$.

The existence of such an expansion follows by the locality property of the affine Kac-Moody vertex algebra (see \cite[Theorem 4.5.2]{BZF} for the general proof for vertex algebras). Indeed, introduce the ``fields'' $A_i(z):=\sum_{l\in\mathbf{Z}}e_i^{(l)}z^{-l-1}$. These fields satisfy locality property $(z-w)^2[A_i(z),A_j(w)]=0$ and the operator product expansion (OPE)
$$A_i(z)A_j(w) = \frac{c_{ij}}{(z-w)^2}+\sum_qf_{ij}^q\frac{A_q(w)}{z-w}+:A_i(z)A_j(w):$$
Here, $c_{ij}$ are scalars corresponding to central charge $c$ (so $c_{ij}=0$ for us), $f_{ij}^q$ are structure constants for $\mathfrak{g}$, and $:A_i(z)A_j(w):$ is the normally ordered product, which lives in $\mathfrak{g}^{\otimes 2}\otimes k[[z,w]][z^{-1},w^{-1}]$.

 Now, rewrite Equation \ref{expansion of poles} as 
$$\omega_r=\sum_{i_1,\dots,i_r}\lambda(A_{i_1}(z_1)\cdots A_{i_r}(z_r))e^{i_1}\otimes\dots\otimes e^{i_r}dz_1\dots dz_r.$$
Consequently, locality and OPE imply $\omega_r\prod_{i<j}(z_i-z_j)^2$ is the expansion of a symmetric element of $(\mathfrak{g}^*)^{\otimes r}\otimes k[[z_1,\dots, z_r]][z_k^{-1}]dz_1\dots dz_r$. And, the residue constraint \ref{residue constraint} follows from the OPE. So, $(\omega_r)\in\Omega_n(D_x^\times)$. Finally, it is clear $\psi_K$ is injective: if $\omega_r$ vanish for all $r$ and $z_i$, then $\phi=0$.

Next, it is also clear the map $\psi_K':\Omega_n(D_x^{\times})\rightarrow (U^{\leq n}(\mathfrak{g}_K))^*$ induced by the residue pairing $\Phi_K$ is the inverse to $\psi_K$. Indeed, this boils down to formula \ref{expansion of poles}. Thus $\psi_K$ being injective and $\psi_K'\circ\psi_K=1$ implies $\psi_K$ is isomorphism.\end{proof}

\begin{prop}\label{differential operators and forms on K/O}\cite{BD}
Let $k$ be an arbitrary field. Let $\phi\in G(K)$ and let $\mathbf{M}_{\phi,n}$ denote the $n^{th}$-filtered component of the vacuum module (see Definition \ref{vacuum module}). There is a perfect pairing induced by restricting Equation \ref{residue pairing on K} to $\Omega_{n}(D_x)\subset\Omega_n(D_x^\times)$
$$\Phi_{K/O}: \mathbf{M}_{\phi,n}\times \Omega_n(D_x)\rightarrow k.$$
\end{prop}

\begin{proof}
First, observe the restriction of $\Phi_K$ to $\Omega_n(D_x)\subset\Omega_n(D_x^\times)$ factors through $\mathbf{M}_{\phi,n}$. Indeed, we must show:
$$\text{If } \omega=(\omega_k\vert_{(D_x^\times)^k})\in\Omega_n(D_x) \text{ and }\xi_1\in\text{Ad}_\phi(\mathfrak{g}_O),\xi_2,\dots,\xi_n\in\mathfrak{g}_K,\text{ then } \Phi_K(\xi_1\cdots\xi_k,\omega)=0.$$
Recall $\phi\in\mathfrak{g}_K$ is the transition function associated to the $G$-bundle $P$, and that Equation \ref{image of trivialization} says $\omega_k\vert_{(D_x^\times)^k}\in(\text{Ad}_\phi\mathfrak{g}^*_O)^{\hat{\otimes} n}[(z_i-z_j)^{-1}, i,j\leq n]dz_1\cdots dz_n$. The standard pairing $\langle,\rangle:\mathfrak{g}_K\times\mathfrak{g}^*\otimes Kdz\rightarrow 0$ is clearly $\text{Ad}_\phi$-invariant, in the sense that 
$$\text{Res}_{z=0}\langle\text{Ad}_{\phi}(\xi),\text{Ad}_\phi\omega\rangle = \text{Res}_{z=0}\langle \xi,\omega\rangle.$$
Let $\tilde{\omega}_1:=\text{Res}_{z_2=0}\dots\text{Res}_{z_k=0}\langle \xi_1(z_1)\otimes\dots\otimes\xi_k(z_k),\omega_k\vert_{(D_x^\times)^k}\rangle \in \text{Ad}_\phi(\mathfrak{g}^*_O)dz_1$. Then $\langle\xi_1,\tilde{\omega_1}\rangle\in k[[z_1]]dz_1$, which implies its residue at $z_1$ is zero. This proves existence of the pairing $\Phi_{K/O}$ as claimed in the Proposition. Thus restricting $\psi_K'$, from the proof of Proposition \ref{differential operators and forms on K}, to $\Omega_n(D_x)\subset\Omega_n(D_x^\times)$ produces a map $\psi_{K/O}':\Omega_n(D_x)\rightarrow \mathbf{M}_{\phi,n}$.

To show $\Phi_{K/O}$ is perfect, it just remains to show restricting $\psi_K:(U^{\leq n}\mathfrak{g}_K)^*\rightarrow\Omega_n(D_x^\times)$ to $\mathbf{M}_{\phi,n}$ has image landing in $\Omega_n(D_x)$. More precisely, define
$$\psi_{K/O}:=\text{Ad}_\phi\circ\psi_K\vert_{\mathbf{M}_{\phi,n}^*}:\mathbf{M}_{\phi,n}^*\rightarrow \Omega_n(D_x^\times).$$
where
$$\text{Ad}_\phi(\omega_r(z_1,\dots,z_r)):=(\text{Ad}_{\phi(z_1)}\otimes\dots\otimes\text{Ad}_{\phi(z_r)})(\omega_r(z_1,\dots,z_r)).$$
Then as we saw before, Equation \ref{expansion of poles} implies $\Omega_n(D_x)\xrightarrow{\psi_{K/O}'}\mathbf{M}_{\phi,n}\xrightarrow{\psi_{K/O}}\Omega_n(D_x^\times)$ is the identity map.

Suppose $\lambda\in\mathbf{M}_{\phi,n}^*$. Then 
$$\text{Ad}_\phi.\lambda(e_{i_1}^{(l_1)}\dots e_{i_r}^{(l_r)})=\lambda(\text{Ad}_\phi(e_{i_1}^{(l_1)})\dots\text{Ad}_\phi(e_{i_r}^{(l_r)}))=0$$
 for all $l_1\geq 0$ because $e_{i_1}^{(l_1)}\in\mathfrak{g}_O$ in that case. Thus, no negative powers of $z_{1}$ appear in $\omega_r$. Then $S_r$-invariance of $\omega_r$ implies no negative powers of $z_i$ appear for any $i$. Hence $\psi_{K/O}(\lambda)\in\Omega_n(D_x)$. \end{proof}

Finally, we prove the global analogue:

\begin{thm}\label{differential operators and forms on Sigma}
Let $k$ be an arbitrary field. Let $\mathbf{M}_{\phi,n}^{out}$ denote $n^{th}$ filtered component of the space of coinvariants (see Definition \ref{vacuum module}). Then there is a perfect pairing 
$$\Phi_{X}: \mathbf{M}_{\phi,n}^{out}\times \Omega_n(X)\rightarrow k$$
induced by the pairing: given $\xi_1,\dots,\xi_k\in \mathfrak{g}_K$ and $\omega:=(\omega_0,\dots,\omega_n)\in\Omega_n(X)$, 
$$\Phi_{X}(\xi_1\dots\xi_k,\omega):= \text{Res}_{z_1=0}\dots\text{Res}_{z_k=0}\langle \xi_1(z_1)\otimes\dots\otimes\xi_k(z_k),\omega_k|_{(D_x^{\times})^k}\rangle$$
As a corollary of Corollary \ref{jets and differential operators},
$$J_P^{n,PD}(\mathcal{M})\simeq \Omega_n(X).$$
\end{thm}

\begin{proof}
Using Proposition \ref{differential operators and forms on K/O}, we just need to show restricting $\psi_{K/O}'$ to $\Omega_{n}(X)$ has image landing in $(\mathbf{M}_{\phi,n}^{out})^*$, and restricting $\psi_{K/O}:\mathbf{M}_{\phi,n}^*\rightarrow\Omega_n(D_x)$ to $(\mathbf{M}_{\phi,n}^{out})^*$ has image landing in $\Omega_n(X)$. 

Let us first discuss $\psi_{K/O}'$. Suppose $\omega=(\omega_k)\in\Omega_n(X)$ and suppose $\xi_1,\dots, \xi_{k-1}\in\mathfrak{g}_K,\xi_k\in\mathfrak{g}_{\text{out}}$. Then $\langle\xi_1\dots\xi_k,\omega\rangle\;\;\text{ is regular on } (D_x^\times)^{\times (k-1)}\times (X\setminus x)$. So on the final component, it may only possibly have a residue at $x$. As $X$ is projective, this forces
$$\text{Res}_{z_k=0}\langle\xi_1\dots\xi_k,\omega\vert_{(D_x^\times)^k}\rangle =0.$$
Next, if we assume $\xi_1\in\text{Ad}_\phi(\mathfrak{g}_O),\xi_2,\dots,\xi_k\in\mathfrak{g}_K$, then we already saw in Proposition \ref{differential operators and forms on K/O} that $\Phi_{X}(\xi_1\dots\xi_k,\omega)=0$. Thus, $\psi_{X}':=\psi_{K/O}'\vert_{\Omega_n(X)}:\Omega_n(X)\rightarrow(\mathbf{M}_{\phi,n}^{out})^*$.

Now, suppose $\lambda\in(\mathbf{M}_{\phi,n}^{out})^*$, and define $\psi_{K/O}(\lambda):=(\omega_k)$ as in Proposition \ref{differential operators and forms on K/O}. There, we showed $\omega_k$ is regular on $(D_x)^{\times k}$. Now, using $\lambda$ kills $\phi\mathfrak{g}_O\phi^{-1}$ on the left-most factor and $\mathfrak{g}_{out}$ on the right-most factor, we conclude $\omega_k$ may be extended to $(D_x)^{\times (k-1)}\times X$ by the Strong Residue Theorem \ref{SRT}. Since $\omega_k$ is $S_k$-invariant, it is thus regular on $(D_x)^{\times i}\times X\times (D_x)^{\times(k-i-2)}$ for all $i$. We will conclude it is then automatically regular on $X^k$ by the following remarkable fact from algebraic geometry, which is a generalization of the strong residue theorem:

\begin{lemma}\label{extending section from tubes to plane} \footnote{This lemma for $m=2$ appears as Theorem 10.3.3 of \cite{BZF}, but there are a few typos present in their proof, and we could not figure out how to correct them. That is why we produced a new independent proof.}
Suppose $X$ is a smooth projective curve over a field $k$, and fix a point $x_0\in X$. Suppose $\mathcal{E}$ is a locally free sheaf on $X^m$, $m\geq 2$, and suppose $s$ is a section of $\mathcal{E}$ defined locally on 
$$s\in\Gamma(\bigcup_{1\leq i\leq m} (D_{x_0})^{\times i}\times X\times (D_{x_0})^{\times(m-i-2)},\mathcal{E}).$$ 
Then $s$ extends uniquely to a global section of $\mathcal{E}$.
\end{lemma}
\begin{proof}
(The following proof was communicated by Sasha Beilinson). The proof will proceed by induction. Let us prove the $m=2$ case first.

Since $\mathcal{E}$ is locally free over $X^2$, we may find a very ample line bundle $\mathcal{L}$ over $X$ such that $\mathcal{E}^*\otimes(\mathcal{L}\boxtimes\mathcal{L})$ is generated by global sections. This means the map
$$V\otimes\mathcal{O}_{X\times X}\rightarrow \mathcal{E}^*\otimes(\mathcal{L}\boxtimes\mathcal{L}),\;\; V:=\Gamma(X^2,\mathcal{E})\otimes(\mathcal{L}\boxtimes\mathcal{L})$$
is surjective. Taking the dual, we find $\mathcal{E}\hookrightarrow V^*\otimes(\mathcal{L}\boxtimes\mathcal{L})$ is injective. Thus, we see the statement for $\mathcal{L}\boxtimes\mathcal{L}$ implies the statement for $\mathcal{E}$. Furthermore, we may assume $\mathcal{L}$ has large enough degree so that $H^1(X,\mathcal{L})=0$. 

Pick $n$ sufficiently large so that $H^0(X,\mathcal{L}(-nx_0))=0$. Let $D_n:=\text{Spec}(\mathcal{O}_X/m_{x_0}^{n+1})$ be the $n^{th}$ infinitesimal neighborhood of $x_0$. Consider the short exact sequence of sheaves on $X$:
\begin{equation}\label{ample exact sequence}
0\rightarrow\mathcal{L}(-nx_0)\rightarrow\mathcal{L}\rightarrow\mathcal{L}\vert_{D_n}\rightarrow 0.
\end{equation}

Now, consider the external tensor product on the left with $(X,\mathcal{L})$ and then with $(D_n,\mathcal{L}\vert_{D_n})$. The restriction map $D_n\times X\rightarrow X\times X$ induces a map on cohomologies:
$$\begin{tikzcd}[column sep=small]
0 \arrow[r] & {H^0(X^2,\mathcal{L}\boxtimes\mathcal{L})} \arrow[r] \arrow[d]    & {H^0(X\times D_n,\mathcal{L}\boxtimes\mathcal{L}\vert_{D_n})} \arrow[r, "\delta"] \arrow[d]    & {H^1(X\times X,\mathcal{L}\boxtimes\mathcal{L}(-nx_0))} \arrow[r] \arrow[d] & 0 \\
0 \arrow[r] & {H^0(D_n\times X,\mathcal{L}\vert_{D_n}\boxtimes\mathcal{L})} \arrow[r] & {H^0(D_n^2,\mathcal{L}\vert_{D_n}\boxtimes\mathcal{L}\vert_{D_n})} \arrow[r, "\delta"] & {H^1(D_n\times X,\mathcal{L}\vert_{D_n}\boxtimes\mathcal{L}(-nx_0))}     \arrow[r]   & 0
\end{tikzcd}$$
Thus, the obstruction of lifting $s\in H^0(X\times D_n,\mathcal{L}\boxtimes \mathcal{L})$ to $H^0(X^2,\mathcal{L}^{\boxtimes 2})$ is  $\delta(s)\in H^1(X\times X,\mathcal{L}\boxtimes\mathcal{L}(-nx_0))$. By Kunneth, we compute this $H^1$ equals $H^0(X,\mathcal{L})\otimes H^1(X,\mathcal{L}(-nx_0))$. Next, we know $s$ restricted to $D_n\times D_n$ extends to $D_n\times X$, so the restriction of $\delta(s)$ to $H^1(D_n\times X,\mathcal{L}\vert_{D_n}\boxtimes\mathcal{L}(-nx_0))$ vanishes. By Kunneth, this $H^1$ equals $H^0(D_n,\mathcal{L}\vert_{D_n})\otimes H^1(X,\mathcal{L}\vert_{D_n})$. Finally, the restriction map 
$$H^0(X,\mathcal{L})\otimes H^1(X,\mathcal{L}(-nx_0))\rightarrow H^0(D_n,\mathcal{L}\vert_{D_n})\otimes H^1(X,\mathcal{L}(-nx_0))$$
is injective because $H^0(X,\mathcal{L}(-nx_0))=0$. Thus, $\delta(s)=0$ and $s$ extends to a global section of $\mathcal{E}$. This concludes the $m=2$ case.

Now suppose the lemma holds for $X^{m-1}$. We may again reduce to the case $\mathcal{E}=\mathcal{L}^{\boxtimes m}$ for a very ample line bundle $\mathcal{L}$. Suppose a section $s$ of $\mathcal{L}^{\boxtimes m}$ is defined on $X\times D_n^{\times (m-1)}$ and all $m$ permutations. By the induction hypothesis, it may be extended to $X^{m-1}\times D_n$. Then consider the exact diagram
$$\adjustbox{scale=.85,center}{
\begin{tikzcd}[column sep=small]
0 \arrow[r] & {H^0(X^m,\mathcal{L}^{\boxtimes m})} \arrow[r] \arrow[d]                                             & {H^0(X^{m-1}\times D_n,\mathcal{L}^{\boxtimes m-1}\boxtimes\mathcal{L}\vert_{D_n})} \arrow[r, "\delta"] \arrow[d] & {H^1(X^m,\mathcal{L}^{\boxtimes m-1}\boxtimes\mathcal{L}(-nx_0))} \arrow[d] \arrow[r]                & 0 \\
0 \arrow[r] & {H^0(D_n^{ m-1}\times X,\mathcal{L}\vert_{D_n}^{\boxtimes m-1}\boxtimes\mathcal{L})} \arrow[r] & {H^0(D_n^m,\mathcal{L}\vert_{D_n}^{\boxtimes m})} \arrow[r, "\delta"]                                             & {H^1(D_n^{m-1}\times X,\mathcal{L}\vert_{D_n}^{\boxtimes m-1}\boxtimes\mathcal{L}(-nx_0))} \arrow[r] & 0
\end{tikzcd}
}$$
where the top row is induced by the cohomology of the external tensor product of the exact sequence \ref{ample exact sequence} with $(X^{m-1},\mathcal{L}^{\boxtimes m-1})$. The vertical arrows are restricting the first $m-1$ factors to $D_n^{m-1}$. Repeating the $m=2$ argument, we find the restriction of 
$$\delta(s)\in H^1(X^m,\mathcal{L}^{\boxtimes m-1}\boxtimes \mathcal{L}(-nx_n))=H^0(X,\mathcal{L})^{\otimes m-1}\otimes H^1(X,\mathcal{L}(-nx_0))$$
to $H^1(D_n^{m-1}\times X,\mathcal{L}\vert_{D_n}^{\boxtimes m-1}\boxtimes\mathcal{L}(-nx_0))=H^0(D_n,\mathcal{L}\vert_{D_n})^{\otimes m-1}\otimes H^1(X,\mathcal{L}(-nx_0))$
 is $0$ because of the assumption that the restriction of $s$ to $D_n^{\times m}$ extends to $D_n^{m-1}\times X$. Moreover, we know the restriction map on $H^1$ is injective because $\mathcal{L}$ is very ample. Thus, the section $s$ lifts to a global section $s\in\Gamma(X^m,\mathcal{L}^{\boxtimes m})$.
\end{proof}
To complete Theorem \ref{differential operators and forms on Sigma}, we may apply Lemma \ref{extending section from tubes to plane} to the locally free sheaf $\mathcal{E}:=(\mathfrak{g}_P^*)^{\boxtimes k}\otimes \tilde{\Omega}_{X^k}(\text{log}D)$  on $X^k$, and the section $s=\omega_k$. This shows 
$$\psi_X:=\psi_{K/O}\vert_{(\mathbf{M}_n^{out})^*}:(\mathbf{M}_{\phi,n}^{\text{out}})^*\rightarrow\Omega_n(X)$$
as desired, and we are finished.
\end{proof}

We conclude with a concise summary of the main results proven in this section. There is a commutative diagram where each horizontal line is a vector space isomorphism. Moreover, the top row is compatible with the algebra structures. 
\begin{equation}\label{summary of jets}
\begin{tikzcd}
{J_{\phi}^{n,PD}(G(K))} \arrow[r, "\sim"]           & ((U\mathfrak{g}_K)_{\leq n})^* \arrow[r, "\sim"]          & \Omega_n(D_x^\times)    \\
{J_{\phi.G(O)}^{n,PD}(\text{Gr}_G)} \arrow[u] \arrow[r, "\sim"] & (\mathbf{M}_{\phi,n})^* \arrow[u] \arrow[r, "\sim"]              & \Omega_n(D_x) \arrow[u] \\
{J_P^{n,PD}(\mathcal{M})} \arrow[u] \arrow[r, "\sim"] & (\mathbf{M}_{\phi,n}^{\text{out}})^* \arrow[u] \arrow[r, "\sim"] & \Omega_n(X) \arrow[u]  
\end{tikzcd}
\end{equation}

\subsection{Changing the point $x\in X$}\label{2 point uniformization}
There is classical fact that ``1-point coinvariants are isomorphic to the 2-point coinvariants.'' In the case of the Heisenberg Lie algebra, and in characteristic $0$, this is proved in Appendix 9.6 of \cite{BZF}. For $\text{Bun}_G$, this follows from the fact that 1-point uniformization  is isomorphic to the 2-point uniformization. Let us formulate this precisely. Let $x,y\in X$ and consider the vacuum module 
$$\mathbf{M}_{x,y}:= \big(U(\mathfrak{g}_{K_x})\otimes_{U(\mathfrak{g}_{O_x})}1_x\big)\otimes_k \big(U(\mathfrak{g}_{K_y})\otimes_{U(\mathfrak{g}_{O_y})}1_y\big).$$
Let $\mathcal{O}_{out}:=\mathcal{O}(X\setminus \{x,y\})$. Then we have injective (Lie algebra) morphism
$$\mathfrak{g}_{out}:=\mathfrak{g}\otimes\mathcal{O}_{out}\hookrightarrow \mathfrak{g}_{K_x}\oplus\mathfrak{g}_{K_y}.$$
Then define 
$$\mathbf{M}_{x,y}^{out}:=\mathbf{M}_{x,y}\big/ \mathfrak{g}_{out}\cdot\mathbf{M}_{x,y}.$$
The PBW filtration on $\mathbf{M}_{x},\mathbf{M}_y$ naturally induces a tensor-product filtration on $\mathbf{M}_{x,y}$, which in turn induces a filtration on $\mathbf{M}_{x,y}^{out}$.
\begin{thm}
There is an isomorphism of graded vector spaces 
$$\mathbf{M}_x^{out(x)}\xrightarrow{\sim}\mathbf{M}_{x,y}^{out(x,y)}$$
\end{thm}
\noindent As a consequence, $\mathbf{M}_x^{out(x)}\simeq\mathbf{M}_y^{out(y)}$ for all $x,y\in X$. This justifies our notation $\mathbf{M}_n^{out}:=(\mathbf{M}_x^{out(x)})_n$ used in section $4$, and moreover shows the independence of choosing $x\in X$ in defining $\Omega_n$.

\section{Pullback to the Fulton-Macpherson Compactification}\label{operads}
In this section we recall the geometry of the Fulton-Macpherson compactification of the configuration space of $n$-points on a curve, $\hat{X}^n$. The main result is Theorem \ref{pullback of BD is BG}, which relates the global sections of the two sheaves considered by \cite{BD,BG}. Finally, in subsection \ref{Jet spaces and the Lie cooperad}, we provide a relationship between the \cite{BG} sheaves and the Lie operad.
\subsection{Resolution of the diagonal}

Let us briefly recall the geometry of the Fulton-Macpherson compactification of the configuration space which we will use. Let $X$ be a smooth projective curve and $n>1$ a positive integer. Let $\mathring{X}^n\subset X^n$ be the open set of all $n$-tuples of pairwise distinct points of $X$, and let $D:=X^n\setminus\mathring{X}^n$ denote the diagonal divisor. Then there exists a smooth projective variety $\hat{X}^n$ and projective morphism $p:\hat{X}^n\rightarrow X^n$ which is an isomorphism over $\mathring{X}^n$ and such that $\hat{D}:=(p^{-1}(D))_{red}$ is a normal crossing divisor. We refer to \cite{FM94} or \cite{BG} for the construction using a sequence of blowups along diagonals. 

The irreducible components of $\hat{D}$ are $\{\hat{D}_I: |I|\geq 2\}$, and each of them is smooth. By a \textit{ tree}, we mean a graph without loops such that there is exactly one ingoing edge and at least 2 outgoing edges at each vertex of the graph. A connected component of a tree has a unique ingoing external edge. Such a component may consist of a single line, in which case it has no vertices and that line is viewed as the both the outgoing and ingoing external edge. A connected tree with a single vertex is called a \textit{star}. An $[n]$-tree consists of a tree together with a bijection between the set $[n]$ and the set of outgoing external edges of the tree. The symmetric group $S_n$ acts naturally on the set of outgoing edges and we consider two labelings of an $[n]$-tree the same if they are in the same coset under the isotropy group $S_n(T)\subset S_n$ action.

There is a natural ordering among $[n]$-trees, where we say $T\leq T'$ if $T'$ is obtained from $T$ via a sequence of operations consisting of either contraction of an internal edge, or deletion of a star containing an ingoing external edge. There is a natural stratification of $\hat{X}^n$ -- we summarize its main properties in the following proposition.
\begin{prop}\label{strata}
There is a stratification $\hat{X}^n=\sqcup_{T}S_T$ by smooth locally-closed algebraic subvarieties $S_T$ such that 
\begin{enumerate}
\item the strata $\{S_T\}$ are indexed by all $[n]$-trees,
\item The codimension of $S_T$ equals the number of vertices of $T$,
\item $S_T\subset \bar{S}_{T'}\Leftrightarrow T\leq T'$
\item If $T$ consists of $n$ connected components, then $S_T=\mathring{X}^n$, the unique open stratum,
\item If $T$ consists of $[n]\setminus I$ connected components, where one consists of an $I$-star and the rest are just lines, then $\bar{S}_T=\hat{D}_I$, an irreducible divisor.
\item Given a vertex $v\in T$, let $E_v\subset [n]$ denote the subset of all the labels attached to all outgoing external edges of $T$ that come out of vertex $v$. Then
$$\bar{S}_T=\bigcap_{\text{ vertices v in T}}\hat{D}_{E_v}.$$
\end{enumerate}
\end{prop}

We may apply the construction of $\hat{X}^n$ to $X=\mathbf{A}_k$, the affine line over $k$. The group of affine transformations $\text{Aff}$ of the line induces an action on $\hat{X}^n$, which acts freely on $\hat{X}^n\setminus\hat{D}_{[n]}\simeq X^n\setminus D_{[n]}$. We find $\mathbf{P}^{n-2}_k\simeq (\mathbf{A}_k^n\setminus D_{[n]})/\text{Aff}$ and define
$$\hat{\mathbf{P}}_k^{n-2}:=(\hat{\mathbf{A}}_k^n\setminus\hat{D}_{[n]})/\text{Aff}.$$
More generally, given a finite set $I$ and curve $X$, define $X^I$ to be the set of $X$-valued functions on $I$. Then $X^I\simeq X^{\#I}$ and $\hat{X}^I=\hat{X}^{\# I}$. Define $\mathbf{P}_k^I:=(\mathbf{A}_k^I\setminus D_I)/\text{Aff}$ and $\hat{\mathbf{P}}_k^I:=(\hat{\mathbf{A}}_k^I\setminus \hat{D}_I)/\text{Aff}$. We caution that $\mathbf{P}_k^I\simeq \mathbf{P}_k^{\# I-2}$. In what follows, we will omit the index $k$ unless we wish to specify a phenomenon specific to characteristic $p$.

The variety $\hat{\mathbf{P}}^I$ also has a stratification $\hat{\mathbf{P}}^I=\sqcup_T \mathbf{P}_T^{\circ}$, parameterized by connected $I$-trees with at least one vertex, where each $\mathbf{P}_T^{\circ}$ is locally closed with smooth closure $\mathbf{P}_T$. 
\begin{prop}\label{strata for P}
For each connected tree $T$, there are canonical isomorphisms
$$\mathring{\mathbf{P}}_T=\prod_{\text{vertices v}\in T}\mathring{\mathbf{P}}^{I(v)},\;\; \mathbf{P}_T=\prod_{\text{vertices v}\in T}\hat{\mathbf{P}}^{I(v)}$$
where $I(v)$ denotes the set of all outgoing edges at the vertex $v$.
\end{prop}
If $T$ consists of just a line, let $\mathbf{P}_T=\mathring{\mathbf{P}}_T=\{pt\}$. For a tree $T$ with connected components $T_1,\dots, T_r$, let $\mathring{\mathbf{P}}_T=\mathring{\mathbf{P}}_{T_1}\times\dots\times\mathring{\mathbf{P}}_{T_r}$. For any stratum $S_T\subset\hat{X}^n$, there are canonical isomorphisms
\begin{equation}\label{strata decomposition}
\bar{S}_T\simeq \mathbf{P}_T\times\hat{X}^J,\;\; S_T\simeq\mathring{\mathbf{P}}_T\times \mathring{X}^J
\end{equation}
where $J$ denotes the set of connected components of $T$. As a corollary, we find 
$$\hat{D}_I\simeq \hat{\mathbf{P}}^I\times \hat{X}^{[n]/I}.$$

\subsection{Construction of the [BG] Sheaf}
Let $I\subset [n]$ be any subset of size at least 2. Let $[n]/I:=[n]\setminus I \cup \{I\}$ be a set of cardinality $n-|I|+1$. Let $\text{Lie}(I)$ denote the Lie operad on $I$. Let $\Omega_{\hat{X}^n,\mathring{X}^n}$,resp. $\Omega(\hat{X}^n,\mathring{X}^n)$, denote sheaf, resp. global sections of the sheaf, of top degree logarithmic differential forms on $\hat{X}^n$ with log poles on $\hat{D}^n$.\footnote{We slightly change notation here to match notation with \cite{BG}}

\begin{prop}\cite[Prop 4.3]{BG}\label{Lie and log} Let $k=\mathbf{Z}$. There is a perfect pairing given by the residue 
 $$\text{Res}:\text{Lie}_k(I)\times \Omega({\hat{\mathbf{P}}_k^I,{\mathring{\mathbf{P}}_k^I}})\rightarrow k.$$
Furthermore, both above $k$-modules are free.
 \end{prop}
The proof of this proposition is modeled by the $I=[3]$ case, for which it then follows from the Jacobi identity together with the property that since there are no global sections of $\Omega_{\mathbf{P}^1}$, log forms $\Omega_{\hat{\mathbf{P}}^{[3]},\mathring{\mathbf{P}}^{[3]}} = \Omega_{\mathbf{P}^1}(\text{log}(D))$ for $D=\{[0:1],[1:0],[1:1]\}$, are determined completely by their residue. Thus,
 $$\Omega(\hat{\mathbf{P}}^{[3]},\mathring{\mathbf{P}}^{[3]}) = \{\lambda_{23}\frac{dz_{23}}{z_{23}}+\lambda_{13}\frac{dz_{13}}{z_{13}}+\lambda_{12}\frac{dz_{12}}{z_{12}}: \lambda_{ij}\in k,\lambda_{23}+\lambda_{13}+\lambda_{12}=0\}$$
 where $z_{ij}=z_i-z_j$. Then the pairing is given by $\text{Res}_{z_{ij}=0}(\omega)=\lambda_{ij}$. And the Jacobi identity on $\text{Lie}(I)$ corresponds to sum of coefficients of $\omega\in \Omega(\hat{\mathbf{P}}^{[3]},\mathring{\mathbf{P}}^{[3]})$ equals $0$.

This pairing of Proposition \ref{Lie and log} provides a canonical distinguished linear map 
$$\phi_I:\mathfrak{g}^{\otimes I}\rightarrow \mathfrak{g}\otimes\Omega({\hat{\mathbf{P}}^I,{\mathring{\mathbf{P}}^I}})$$
characterized by the property that, for all $I$-binary trees $T$, $\text{Res}_T(\Phi_I(\xi_1\otimes\dots\otimes\xi_I))$ consists of a Lie-bracket expression involving insertion $[\xi_i,\xi_j]$ whenever edges $i,j$ share a vertex in $T$. Here, 
\begin{equation}\label{Res_T}
\text{Res}_T:=\text{Res}_{\hat{D}_1\cap\cdots\cap\hat{D}_r}\circ\cdots\circ\text{Res}_{\hat{D}_{r-1}\cap\hat{D}_r}\circ\text{Res}_{\hat{D}_r}
\end{equation}
where $\hat{D}_1,\dots,\hat{D}_r$ are the irreducible diagonal divisors corresponding to each of the $r$ vertices of the binary tree $T$. The order of the $\hat{D}_i$ taken does not affect $\text{Res}_T$ because of normal crossing. Denote the linear adjoint of $\phi_I$ by 
\begin{equation}\label{The fibers phi_I^*}
\phi_I^*:\mathfrak{g}^*\rightarrow(\mathfrak{g}^*)^{\otimes I}\otimes \Omega({\hat{\mathbf{P}}^I,{\mathring{\mathbf{P}}^I}}).
\end{equation}

 We now explain $\phi_I^*$ is $\text{Ad}_g$-invariant for any $g\in G$. Using Proposition \ref{Lie and log}, we may replace $\Omega(\hat{\mathbf{P}}^I,\mathring{\mathbf{P}}^I)$ with $\text{Lie}(I)^*$ since the Residue is $\text{Ad}_g$-equivariant. For simplicity, write $I=[d]$. It is well known $\text{Lie}(I)$ has dimension $(d-1)!$ and a basis is given by $\mathbf{x}_\sigma:=[x_{\sigma(1)},[\dots,[x_{\sigma(d-1)},x_d]\dots]]$ for ${\sigma\in S_{d-1}}$. Then explicitly, 
 $$\phi_I^*:A\mapsto (x_1\otimes\dots\otimes x_d\mapsto A(\mathbf{x}_\sigma))_{\sigma\in S_{d-1}}.$$
This expression is clearly $\text{Ad}_g$-invariant, as desired. 
 
Thus, we can construct a relative version of $\phi_I^*$, denoted by $\Phi_I^*$. By relative, we mean $\Phi_I^*$ is a $\mathcal{O}_{\hat{D}_I}$-module morphism whose fibers are $\phi_I^*$. Let $P$ be a principal $G$-bundle over $X$. Let $\mathfrak{g}_P:=\mathfrak{g}\otimes\mathcal{O}_P/G$ be the associated adjoint bundle on $X$. Let $p:\hat{X}^n\rightarrow X^n$ be the Fulton-Macpherson compactification, and denote $\hat{\mathfrak{g}}_P^{\boxtimes n}:=p^*(\mathfrak{g}_P^{\boxtimes n})$ the associated bundle on $\hat{X}^n$. Similarly, denote $(\hat{\mathfrak{g}}_P)^{\boxtimes [n]/I}:=p^*(\mathfrak{g}_P^{\boxtimes [n]/I})$ for $p:\hat{X}^{[n]/I}\rightarrow X^{[n]/I}$.

 First, we tensor $\phi_I^*$ on the left by $\mathcal{O}_P$ and (external) on the right by $\mathcal{O}_{\hat{\mathbf{P}}^I}$ to obtain the following morphism of $\mathcal{O}_{X\times\hat{\mathbf{P}}^I}$-modules
$$\Phi_I^*:\mathcal{O}_P\otimes\mathfrak{g}^*\boxtimes\mathcal{O}_{\hat{\mathbf{P}}^I}\rightarrow \mathcal{O}_P\otimes(\mathfrak{g}^*)^{\otimes I}\boxtimes\Omega_{\hat{\mathbf{P}}^I,\mathring{\mathbf{P}}^I}.$$
The $\text{Ad}_G$ equivariance allows us to descend to a map of $\mathcal{O}_{X\times\hat{\mathbf{P}}^I}$-modules
\begin{equation}\label{Phi_I^* step 1}
\Phi_I^*:\mathfrak{g}_P^*\boxtimes\mathcal{O}_{\hat{\mathbf{P}}^I}\rightarrow(\mathfrak{g}_P^*)^{\otimes I}\boxtimes\Omega_{\hat{\mathbf{P}}^I,\mathring{\mathbf{P}}^I}.
\end{equation}

Now, we have $\hat{D}_I \simeq \hat{X}^{[n]/I}\times\hat{\mathbf{P}}^I$ and $D_I\simeq X^{[n]\setminus I}\times X$. Consider the following commutative diagram 

\begin{equation}\label{diagonal diagram}
\begin{tikzcd}
  & \hat{D}_I \arrow[d, "p_2"] \arrow[r, "p_1"] \arrow[ldd, "f_1"'] \arrow[rdd, "f_2"] & \hat{\mathbf{P}}^I   \\
  & {\hat{X}^{[n]/I}} \arrow[ld, "\tilde{f}_1"] \arrow[rd, "\tilde{f}_2"']                                            &                      \\
X &                                                                                    & {X^{[n]\setminus I}}
\end{tikzcd}
\end{equation}
Then we have 
\begin{align}\label{isomorphisms along diagonal}
\Omega_{\hat{D}_I,\mathring{D}_I}&\simeq p_1^*(\Omega_{\hat{\mathbf{P}}^I,\mathring{\mathbf{P}}^I})\otimes p_2^*(\Omega_{\hat{X}^{[n]/I},\mathring{X}^{[n]/I}})\\
(\hat{\mathfrak{g}}_P^*)^{\boxtimes n}\vert_{\hat{D}_I} &\simeq f_1^*(\mathfrak{g}_P^*)^{\otimes I}\otimes f_2^*((\mathfrak{g}_P^*)^{\boxtimes [n]\setminus I})\\
p_2^*((\hat{\mathfrak{g}}_P^*)^{\boxtimes [n]/I}) &\simeq f_1^*(\mathfrak{g}_P^*)\otimes f_2^*((\mathfrak{g}_P^*)^{\boxtimes [n]\setminus I})
\end{align}

Next, if we take Equation \ref{Phi_I^* step 1} and pullback along the map $(f_1,p_1):\hat{D}_I\rightarrow X\times\hat{\mathbf{P}}^I$, and using pullback commutes with tensor product, and that pullback along a projection corresponds to taking external tensor product with structure sheaf, we obtain the following morphism of $\mathcal{O}_{\hat{D}_I}$-modules:

\begin{equation}\label{Phi_I^* step 2}
\Phi_I^*:f_1^*(\mathfrak{g}_P^*)\rightarrow f_1^*(\mathfrak{g}_P^*)^{\otimes I}\otimes p_1^*(\Omega_{\hat{\mathbf{P}}^I,\mathring{\mathbf{P}}^I}).
\end{equation}

Next, we tensor the above map by $-\otimes f_2^*((\mathfrak{g}_P^*)^{\boxtimes [n]\setminus I})\otimes p_2^*(\Omega_{\hat{X}^{[n]/I},\mathring{X}^{[n]/I}})$ and use the three identifications in \ref{isomorphisms along diagonal} to obtain the map of $\mathcal{O}_{\hat{D}_I}$-modules \footnote{This corrects a few typos present in Equation 7.3 [BG]}

 \begin{equation}\label{Psi_I^* final}
 \Psi_I^*:p_2^*\big((\hat{\mathfrak{g}}_P^*)^{\boxtimes [n]/I}\otimes\Omega_{\hat{X}^{[n]/I},\mathring{X}^{[n]/I}}\big)\rightarrow (\hat{\mathfrak{g}}_P^*)^{\boxtimes n}\vert_{\hat{D}_I} \otimes \Omega_{\hat{D}_I,\mathring{D}_I}
 \end{equation}
 
 Next, observe that since $\mathfrak{g}$ is semi-simple, the Lie-bracket map is surjective. From this it directly follows $\phi_I^*$ is injective. Then at each step of the construction leading to $\Psi_I^*$, we preserved left-exactness (since we first pulled-back, then took tensor product with a locally-free module). Thus $\Psi_I^*$ is an embedding and its image $\text{Im}(\Psi_I^*)$ is a locally-free sheaf on $\hat{D}_I$ with fiber $(\mathfrak{g}^*)^{[n]/I}$.
  
We are ready to state the main definition:
\begin{definition} Define the ``BG sheaf'' on $\hat{X}^n$ of logarithmic differential forms with prescribed residue along the diagonal
$$\hat{\mathcal{G}}_n:=\{\omega\in(\hat{\g}_P^*)^{\boxtimes n}\otimes\Omega_{\hat{X}^n,\mathring{X}^n}:\;\text{Res}_{\hat{D}_I}(\omega)\in\text{Im}(\Psi_I^*)\}$$
\end{definition}
Note, $S_n$ acts naturally on $X^n$, and this induces, by functoriality, an action on $\hat{X}^n$ making $(\hat{\mathfrak{g}}^*_P)^{\boxtimes n}, \Omega_{\hat{X}^n,\mathring{X}^n},\hat{\mathcal{G}}_n$ all into $S_n$-equivariant sheaves. Next, $S_n$ acts on all $[n]$-trees by permuting the labelings. Given an $[n]$-tree $T$ and $\sigma\in S_n$, we have $\sigma(S_T)=S_{\sigma(T)}$. Thus, $S_n(T)$, the isotropy group of $T$, acts on $S_T$, and we denote:

\begin{definition}\label{antiinvariants of BG sheaf}
Let
\begin{equation*}H^0(\hat{X}^n,\hat{\mathcal{G}}_n)^{-S_n}:=
									\bigg\{  s\in H^0(\hat{X}^n,\hat{\mathcal{G}}_n) :
									\begin{aligned}
									& \text{For each strata } S_T \text{ and } \sigma\in S_T,\\
									&\text{we have }\sigma\text{Res}_T(s)=\text{sign}(\sigma)\cdot\text{Res}_T(s)\\
									\end{aligned}
									\bigg\}
									\end{equation*}
\end{definition}

Now, we explain how to upgrade the construction of $\Psi_I^*$ to $\Psi_T^*$, where $T$ is a tree. Let $I_1,I_2 \subset [n]$ and suppose $i\in I_1$. Let $I:= I_1\circ_i I_2:=(I_1\setminus i)\sqcup I_2$. Given a connected $I_1$-grove $T_1$ and connected $I_2$-grove $T_2$, each with a single vertex (i.e ``stars''), we may define $T:=T_1\circ_i T_2$ by inserting at the outgoing edge $i$ of $T_1$, the unique ingoing external edge of $T_2$. We caution that $T_2\circ T_1\neq T_1\circ T_2$ in general. Let $\hat{D}_{I_1},\hat{D}_{I_2}$ be the corresponding irreducible divisors. We know the corresponding stratum $\bar{S}_T=\hat{D}_{I_1}\cap\hat{D}_{I}$. 
By Proposition \ref{strata for P} and Equation \ref{strata decomposition}, we know 
$$\Omega_{\mathbf{P}_T,\mathring{\mathbf{P}}_T}\simeq\Omega_{\hat{\mathbf{P}}^{I_1},\mathring{\mathbf{P}}^{I_1}}\boxtimes\Omega_{\hat{\mathbf{P}}^{I_2},\mathring{\mathbf{P}}^{I_2}},\;\;\text{ and } \Omega_{\bar{S}_T,S_T} = \Omega_{\mathbf{P}_T,\mathring{\mathbf{P}}_T}\boxtimes\Omega_{\hat{X}^{[n]/I},\mathring{X}^{[n]/I}}.$$

Consider the following commutative diagrams involving $S_T$:
$$
\begin{tikzcd}
  & \bar{S}_T \arrow[d, "{p_{T,2}}"] \arrow[ldd, "f_1"'] \arrow[rdd, "f_2"]  &                      &                        & \bar{S}_T\arrow[r, "{i_T}"]\ \arrow[d, "{p_{T,1}}"] \arrow[ldd, "p_1"'] \arrow[rdd, "p_2"] &       \hat{D}_I &                 \\
  & {\hat{X}^{[n]/I}} \arrow[ld] \arrow[rd]                           &                      &                        & \mathbf{P}_T \arrow[ld] \arrow[rd]                                &                        \\
X &                                                                   & {X^{[n]\setminus I}} & \hat{\mathbf{P}}^{I_1} &                                                                   & \hat{\mathbf{P}}^{I_2}
\end{tikzcd}
$$

Thus, taking Equation \ref{Phi_I^* step 2} for $I_1$, restricting to $S_T\hookrightarrow\hat{D}_{I_1}$, and then applying $\circ_i\Phi^*_{I_2}$ (i.e apply $\Phi_{I_2}^*$ at the $i$th copy of $f_1^*(\mathfrak{g}_P^*)$ in $f_1^*(\mathfrak{g}_P^*)^{\otimes I_1}$), we obtain a morphism of sheaves on $\bar{S}_T$

\begin{equation}\label{Phi_T^* plain}
\Phi_T^*:f_1(\mathfrak{g}_P^*)\xrightarrow{\Phi_{I_1}^*}f_1^*(\mathfrak{g}_P^*)^{\otimes I_1}\otimes p_1^*(\Omega_{\hat{\mathbf{P}}^{I_1},\mathring{\mathbf{P}}^{I_1}})\xrightarrow{_i\circ\Phi_{I_2}^*}f_1^*{(\mathfrak{g}_P^*)^{\otimes I}}\otimes p_{T,1}^*(\Omega_{\mathbf{P}_T,\mathring{\mathbf{P}}_T})
\end{equation}

\noindent (Recall $I=I_1\circ_i I_2$) Finally, tensor both sides by $f_2^*((\mathfrak{g}_P^*)^{\boxtimes [n]\setminus I})\otimes p_{T,2}^*(\Omega_{\hat{X}^{[n]/I},\mathring{X}^{[n]/I}})$ to obtain the desired morphism of $\bar{S}_T$-modules

\begin{equation}\label{Psi_T^* final}
\Psi_T^*:p_{T,2}^*\big((\hat{\mathfrak{g}}_P^*)^{\boxtimes [n]/I}\otimes\Omega_{\hat{X}^{[n]/I},\mathring{X}^{[n]/I}}\big)\rightarrow (\hat{\mathfrak{g}}_P^*)^{\boxtimes n}\vert_{\bar{S}_T} \otimes \Omega_{\bar{S}_T,{S}_T}
\end{equation}

Again, since $\mathfrak{g}$ is semisimple, this morphism is a locally-split embedding, and its image is a locally free sheaf on $\bar{S}_T$ with fibers $(\mathfrak{g}^*)^{\otimes I}$. We make the remark that $\Psi_T^*$ may only be defined on $\bar{S}_T$, for the composition of $\Phi_{I_1}^*$ and $\Phi_{I_2}^*$  is only well-defined on the intersection of their supports, i.e $\bar{S}_T$. 

Next, observe the key defining property of $\Psi_I^*$ is the commutative diagram of sheaves on $\bar{S}_T$ for all $I$ trees $T$ (still keeping the $I=I_1\circ_i I_2\circ\cdots, T=T_1\circ_i T_2\circ\cdots$ notation, and where $\text{Res}_T$ is as defined in Equation \ref{Res_T}):

$$\begin{tikzcd}
{p_{T,2}^*((\hat{\mathfrak{g}}^*_P)^{\boxtimes [n]/I}\otimes\Omega_{\hat{X}^{[n]/I},\mathring{X}^{[n]/I}}}) \arrow[rd, "\Psi_T^*"'] \arrow[r, "\Psi_I^*"] & {(\hat{\mathfrak{g}}_P^*)^{\boxtimes n}\vert_{\bar{S}_T}\otimes\Omega_{\hat{D}_I,\mathring{D}_I}\vert_{\bar{S}_T}} \arrow[d, "\text{Res}_{T}"] \\
                                                                                                                                                   & {(\hat{\mathfrak{g}}_P^*)^{\boxtimes n}\vert_{\bar{S}_T}\otimes\Omega_{\bar{S}_T,S_T}}                                                                                
\end{tikzcd}$$
In summary, we have
\begin{equation}\label{compatibility}
\text{Res}_{T}\circ(\Psi_{I}^*\vert_{S_T}) = \Psi_T^*\;\;\;\text{ for all I-trees T.} \end{equation}

Next, let $I$ be a $d$ element subset of $[n]$, which for simplicity of notation we suppose $I=[d]$. Let $T_0$ denote that binary $I$-tree $T_0=[1,[2,\dots, [d-1,d]\dots]$, and let $\bar{S}_{T_0}:=\hat{D}_{[d]}\cap\hat{D}_{\{2,3,\dots, d\}}\dots\cap\hat{D}_{\{d-1,d\}}$ denote the corresponding strata in $\hat{D}_{[d]}$, and denote by $p_0=p_{T_0,1}(S_{T_0})$ be the corresponding point in $\hat{\mathbf{P}}^I$ so that $\bar{S}_{T_0} \simeq\hat{X}^{[n]/I}\times\{p_0\}$. Next, for $\sigma\in S_{d-1} = \text{Stab}_{S_d}(d)$, denote the binary $I$-tree $T_\sigma:=\sigma.T_0 = [\sigma(1),[\sigma(2),\dots, [\sigma(d-1),d]\dots]]$. Similarly, let $p_\sigma:=\sigma(p_0)\in\hat{\mathbf{P}}^I$ be the corresponding point. Let $i_\sigma:\{p_\sigma\}\hookrightarrow\hat{\mathbf{P}}^I$ denote the inclusion. Finally, let $\text{adj}_T:\mathcal{F}\rightarrow (i_T)_*i_T^*(\mathcal{F})$ be the unit map for $\mathcal{F}$ a sheaf on $\hat{\mathbf{P}}^I$, and denote the locally free $\hat{X}^{[n]/I}$-modules by
 $$\mathcal{F}_{[n]/I}:=(\hat{\mathfrak{g}}_P^*)^{\boxtimes [n]/I}\otimes\Omega_{\hat{X}^{[n]/I},\mathring{X}^{[n]/I}}.$$

$$\mathcal{E}^{[n]/I}:= \tilde{f}_1(\mathfrak{g}_P^*)^{\otimes I}\otimes \tilde{f}_2((\mathfrak{g}_P^*)^{\boxtimes [n]\setminus I})\otimes\Omega_{\hat{X}^{[n]/I},\mathring{X}^{[n]/I}}$$

\noindent Thus, we may rewrite the compatibility equation \ref{compatibility} as a commutative diagram of $\hat{D}_I=\hat{X}^{[n]/I}\times\hat{\mathbf{P}}^I$-modules:

\begin{equation}\label{compatibility diagram}
\begin{tikzcd}
{\mathcal{F}_{[n]/I}\boxtimes\mathcal{O}_{\hat{\mathbf{P}}^I}} \arrow[r, "\Psi_I^*"] \arrow[d, "(\text{adj}_{T_\sigma})"'] & {\mathcal{E}_{[n]/I}\boxtimes\Omega_{\hat{\mathbf{P}}^I,\mathring{\mathbf{P}}^I}} \arrow[d, "(\text{Res}_{T_\sigma})"] \\
{\bigoplus_{\sigma\in S_{d-1}}\mathcal{F}_{[n]/I}\boxtimes(i_{\sigma})_*(\{p_\sigma\})} \arrow[r, "(\Psi_{T_\sigma}^*)"']         & {\bigoplus_{\sigma\in S_{d-1}}\mathcal{E}_{[n]/I}\boxtimes(i_{\sigma})_*(\{p_\sigma\})}                                                          
\end{tikzcd}
\end{equation}
On global sections, the left vertical map is the diagonal embedding, the right vertical is an isomorphism, by Proposition \ref{Lie and log}, and the horizontal maps are embeddings.

\subsection{The pullback}
\begin{definition}\label{BD sheaf}
Define the ``BD sheaf'' $\mathcal{G}_n$ over $X^n$ whose local sections are:
 \begin{align*}
\mathcal{G}_n:=\{\omega\in(\mathfrak{g}_P^*)^{\boxtimes n}\otimes\tilde{\Omega}_{X^n}(\text{log}(D)) :\text{For all $ d\leq n$ and $\sigma\in S_{d-1}$,\;}\text{Res}_{T_{\sigma}}(\omega)\in\text{Im}(\Psi_{T_{\sigma}}^*)\}
\end{align*}
where for a binary $d$-tree $T_{\sigma}=[\sigma(1),[\dots,[\sigma(d-1),\sigma(d)]\dots]$, we denote 
$$\text{Res}_{T_{\sigma}}:=\sigma(\text{Res}_{D_{[d]}}\circ\text{Res}_{D_{\{2,\dots, d\}}}\circ\dots\circ\text{Res}_{D_{\{d-1,d\}}}),$$ and $\Psi_{T_{\sigma}}^*$ is the corresponding Lie cobracket expression as defined in Equation \ref{Psi_T^* final}.
\end{definition}
\noindent Note, we denote $\Psi_T^*$ for both the map of sheaves on $X^n$ and on $\hat{X}^n$ because the latter is induced by the former under pullback, and they are both induced by the same corresponding cobracket expression on fibers. Also, analogues to Definition \ref{antiinvariants of BG sheaf}, we may define $H^0(X^n,\mathcal{G}_n)^{-S_n}$. Since $p:\hat{X}^n\rightarrow X^n$ is a proper map, the global sections of $\mathcal{G}_n$ are equal to that $p^*\mathcal{G}_n$. The sheaf $\mathcal{G}_n$ over $X^n$ is related to the universal sheaf $\Omega_n$ over $X$ (recall Definition \ref{The universal space}) in the following way:
$$\Gamma(X,\Omega_n)\xrightarrow{\sim}\Gamma(X^n,\mathcal{G}_n)^{-S_n}\xleftarrow{\sim}\Gamma(\hat{X}^n,p^*\mathcal{G}_n)^{-S_n}$$
where first map is $\omega=(\omega_i)_{0\leq i\leq n}\mapsto\omega_n$.

Next, we wish to relate $p^*\mathcal{G}_n$ with $\hat{\mathcal{G}}_n$, where $p:\hat{X}^n\rightarrow X^n$ is the Fulton-Macpherson compactification. A simple, yet important, observation is that on the smooth locus of $D$,
\begin{equation}\label{first compatibility of residue pullback} 
\text{Res}_{\hat{D}}(p^*\omega)\vert_{\hat{D}_{ij}} =p^*( \text{Res}_{D}(\omega)\vert_{D_{ij}})\vert_{\hat{D}_{ij}}
\end{equation}
This follows because the singular locus of $D$ is $D^{\text{sing}} = \bigcup_{ijk} D_{ijk}$ and $p$ is an isomorphism on $p^{-1}(X^n\setminus D^{\text{sing}}) = \hat{X}^n\setminus\bigcup_{I\subset [n]: |I|\geq 3} \hat{D}_I.$

Following the same reasoning \footnote{For instance, the regular locus $D_{123}^{\text{reg}}$ inside $D_{12}$ is $D_{123}\setminus \bigcup_{i\geq 4}D_{1234}$, and this equals the regular locus $(\hat{D}_{12}\cap\hat{D}_{123})^{\text{reg}}=\hat{D}_{123}\setminus \bigcup_{[3]\subset I, |I|\geq 4}\hat{D}_I$ inside $\hat{D}_{12}$.}, we find in general that for any tree of the form $T_\sigma$, 
\begin{equation}\label{next compatibility of residue pullback}\text{Res}_{{T}_\sigma}(p^*\omega) =\lambda_\sigma p^*( \text{Res}_{T_\sigma}(\omega))\vert_{S_{T_\sigma}}.
\end{equation}
where $\text{Res}_{T_{\sigma}}$ is computed using Definition \ref{Res_T} on the left and Definition \ref{BD sheaf} on the right. The scalar $\lambda_\sigma\in\mathbf{Z}_{>0}$ appears because $\hat{D}_I$, for $|I|\geq 3$, appears with some multiplicity inside the non-reduced locus $p^{-1}(D)$, and there is the remarkable property that $\frac{dh^\alpha}{h^\alpha}=\alpha\frac{dh}{h}$. Note that in our case, local sections $\omega\in\mathcal{G}_n$ may be written as $\omega=\sum_{\sigma\in S_n}a_\sigma\omega_\sigma$  for $\omega_\sigma\in(\mathfrak{g}_P^*)^{\boxtimes n}\otimes\Omega_{X^n}(D_\sigma)$. And, $\text{Res}_{T_\sigma}(\omega)=a_\sigma\text{Res}_{T_\sigma}(\omega_\sigma)$. So, Equation \ref{next compatibility of residue pullback} is really a statement about forms with poles along a normal crossing divisor, and consequently there is no ambiguity in the order in which we take $\text{Res}_{T_\sigma}(\omega)$ downstairs. We will suppress this integer $\lambda_\sigma$ from notation as it does not affect the results.

\begin{lemma}\label{pullback lands in BG}
Suppose $\omega\in\Gamma(X^n,\mathcal{G}_n)^{-S_n}$. Then $p^*\omega\in\Gamma(\hat{X}^n,\hat{\mathcal{G}}_n)^{-S_n}$.
\end{lemma}
\begin{proof}
First, Lemma \ref{pullback of NCD} (and see Remark \ref{justify NCD log forms}) shows
$$p^*\omega\in\Gamma(\hat{X}^n,(\hat{\mathfrak{g}}_P^*)^{\boxtimes n}\otimes\Omega_{\hat{X}^n}(\text{log}(\hat{D})))^{-S_n}.$$
Thus, it remains to check $\text{Res}_{\hat{D}_I}(p^*\omega)\in\text{Im}(\Psi_I^*)$ for all $I\subset [n]$. By symmetry, it suffices to check for $I=[d]$ for each $2\leq d\leq n$. A summary of the following proof is that there are no global sections of $\Omega_{\hat{\mathbf{P}}^I}$, so logarithmic differential forms $\Omega({\hat{\mathbf{P}}^I,\mathring{\mathbf{P}}^I})$ are completely determined by all possible residues $\text{Res}_T(\omega)$, for $T$ a $I$-binary tree, in the sense that there exists a unique log form with those prescribed residues.

We know by \ref{next compatibility of residue pullback} that for any binary $d$-tree of the form $T_\sigma$,
$$\text{Res}_{T_\sigma}(p^*\omega)=p^*\text{Res}_{T_\sigma}(\omega).$$
 This means for each binary $d$-tree of the form $T_\sigma,\sigma\in S_{d-1}$, we have 
$$\text{Res}_{T_\sigma}(p^*\omega)\in\text{Im}(\Psi_{T_\sigma}^*).$$
Also, note the because of transverse intersection, the order in which we compute the iterated residue in $\text{Res}_{T_\sigma}(p^*\omega)$ does not matter, up to sign. So if we start with $\text{Res}_{\hat{D}_I}(p^*\omega)\in \mathcal{E}_{[n]/I}\boxtimes\Omega_{\hat{\mathbf{P}}^I,\mathring{\mathbf{P}}^I}$, then from diagram \ref{compatibility diagram}, we can conclude its image under right vertical arrow comes from the image of the left horizontal arrow. Namely, there exist $A_\sigma\in\mathcal{F}_{[n]/I}$ such that 
\begin{equation}
\text{Res}_{T_\sigma}(p^*\omega) = \Psi_{T_\sigma}^*(A_\sigma).
\end{equation}

Next, we show $S_n$-anti-invariance of $p^*\omega$ implies $A_\sigma=\text{sign}(\sigma)A_0$ for all $\sigma$. Indeed:
\begin{equation}\label{A_sigma and A_0}
\Psi_{T_\sigma}^*(A_\sigma)=\text{Res}_{T_\sigma}(p^*\omega) = \text{sign}(\sigma).\sigma\text{Res}_{T_0}(p^*\omega)=\text{sign}(\sigma).\sigma\Psi_{T_0}^*(A_0).
\end{equation}

Next, recall $\Psi_{T_\sigma}^*$ was defined so that its fibers are the maps
$$\psi_{T_\sigma}^*:\mathfrak{g}^*\rightarrow (\mathfrak{g}^*)^{\otimes I},\;\; A_\sigma\mapsto (x_1\cdots x_d\mapsto A_\sigma([x_{\sigma(1)},\dots, [x_{\sigma(d-1)},x_d]\dots]))$$
So, $\Psi_{T_\sigma}^*(A_\sigma) = \sigma\Psi_{T_0}^*(A_\sigma)$ and consequently equation \ref{A_sigma and A_0} implies
$$\Psi_{T_0}^*(A_\sigma)=\Psi^*_{T_0}(\text{sign}(\sigma)A_0).$$
But, $\mathfrak{g}$ is semsimple, so $\Psi^*_{T_0}$ is injective, thus $A_\sigma=\text{sign}(\sigma)A_0$ for all $\sigma$. The left vertical arrow in \ref{compatibility diagram} is just the diagonal inclusion (up to signs), thus we conclude $A_0\in\mathcal{F}_{[n]/I}\boxtimes\mathcal{O}_{\hat{\mathbf{P}}^I}$ is such that $A_0\rightarrow (A_\sigma)_{\sigma\in S_{d-1}}$. Finally, commutativity of the diagram and injectivity of both horizontal arrows forces $\text{Res}_{\hat{D}_I}(p^*\omega)=\Psi_I^*(A_0)$, as desired!
\end{proof}

\begin{lemma}\label{Residue as map of jets}
We have a morphism of $\mathcal{O}_{\hat{X}^n}$-modules
$$\hat{R}_n:=(\Psi_{\{n-1,n\}}^*)^{-1}\circ\text{Res}_{\hat{D}_{\{n-1,n\}}}:\hat{\mathcal{G}}_n\rightarrow(\Delta_{n-1,n})_*(\hat{\mathcal{G}}_{n-1})$$
where $\Delta_{n-1,n}:\hat{D}_{n-1,n}\hookrightarrow\hat{X}^n$ is the inclusion.   
\end{lemma}
\begin{proof}
The inverse of $\Psi^*_I$ is well-defined on $\text{Res}_{\hat{D}_{I}}\hat{\mathcal{G}}_n$ since it is an isomorphism on it. Next, observe
$$\{I\subset [n]: \hat{D}_I\cap \hat{D}_{n-1,n}\neq\o, |I|\geq 2\} = \{I\circ_i\{n-1,n\} : i\in I\text{ and } I\subset [n-1], |I|\geq 2\}.$$
Thus, we must check $\text{Res}_{\bar{S}_T}\text{Res}_{\hat{D}_I}(\omega)\in\text{Im}(\Psi_T^*)$ for all sections $\omega\in\hat{\mathcal{G}}_n$. Here,  $\tilde{I}=I\circ_i\{n-1,n\}, I\subset[n-1],$ and $T$ is the $\tilde{I}$-tree obtained by inserting $\{n-1,n\}$ at vertex $i$ of the connected $I$-star.

Since $\omega\in\hat{\mathcal{G}}_n$, there exists $\omega_{\tilde{I}}$ such that $\Psi_{\tilde{I}}^*(\omega_{\tilde{I}})=\text{Res}_{\hat{D}_{\tilde{I}}}\omega$ and $\omega_{n-1,n}$ such that $\Psi_{n-1,n}^*(\omega_{n-1,n})=\text{Res}_{\hat{D}_{n-1,n}}(\omega)$. Transverse intersection and Equation \ref{compatibility}, imply
$$\text{Res}_{\bar{S}_T}(\Psi_{n-1,n}^*(\omega_{n-1,n}))=\text{Res}_{\bar{S}_T}\text{Res}_{\hat{D}_{n-1,n}}(\omega) = \text{Res}_{\bar{S}_T}\text{Res}_{\hat{D}_I}(\omega)=\text{Res}_{\bar{S}_T}\Psi_{\tilde{I}}^*(\omega_{\tilde{I}})=\Psi_T^*(\omega_{\tilde{I}})\in\text{Im}(\Psi_T^*),$$
as desired.
\end{proof}

\begin{thm}\label{pullback of BD is BG}
Let $p:\hat{X}^n\rightarrow X^n$ be the Fulton-Macpherson compactification over any algebraically closed field. Then $\Gamma(\hat{X}^n,\hat{\mathcal{G}}_n)^{-S_n}\simeq\Gamma(\hat{X}^n,p^*(\mathcal{G}_n))^{-S_n}$.
\end{thm}
\begin{proof}
Consider the map ${R}_n:=(\Psi_{n-1,n}^*)^{-1}\circ\text{Res}_{{D}_{n-1,n}}:\mathcal{G}_n\rightarrow(\Delta_{n-1,n})_*(\mathcal{G}_{n-1})$ of $\mathcal{O}_{X^n}$-modules. On global sections, it maps anti-invariants to anti-invariants by definition \ref{antiinvariants of BG sheaf}.  The map
$$R_n:\Gamma(X^n,\mathcal{G}_n)^{-S_n}\rightarrow\Gamma(X^{n-1},\mathcal{G}_{n-1})^{-S_{n-1}}$$
has kernel
$$\text{Ker}(R_n)=\Gamma(X^n,(\mathfrak{g}_P^*)^{\otimes n}\otimes\Omega_{X^n})^{-S_n}\simeq (\Gamma(X,\mathfrak{g}_P^*\otimes\Omega_X)^{\otimes n})^{S_n} = S^n(J_P^{1}(\mathcal{M}))$$
where $S^n(J_P^{1}(\mathcal{M}))$ denotes the $n$th symmetric power of the 1st infinitesimal jet space of $\text{Bun}_G$. Indeed, suppose $\omega\in\Gamma(X^n,\mathcal{G}_n)^{-S_n}$ is such that $\text{Res}_{D_{n-1,n}}(\omega)=0$. Then $S_n$-anti-invariance of $\omega$ forces $\text{Res}_{D_{ij}}(\omega)=0$ for all $i,j$, and consequently $\omega$ is regular on $X^n$, thus proving the claim.

Next, consider $\mathcal{O}_{\hat{X}^n}$-modules. Sections of $p^*(S^n(\mathcal{G}_1))$ are regular on $\hat{X}^n$, hence automatically meet the residue constraint. Thus, $p^*(S^n(\mathcal{G}_1))\hookrightarrow\hat{\mathcal{G}}_n$. We claim $\Gamma(\hat{X}^n,p^*(S^n(\mathcal{G}_1)))$ is the kernel of 
$$\hat{R}_n:=(\Psi_{n-1,n}^*)^{-1}\circ\text{Res}_{\hat{D}_{n-1,n}}:\Gamma(\hat{X}^n,\hat{\mathcal{G}}_n)^{-S_n}\rightarrow\Gamma(\hat{X}^{n-1},\hat{\mathcal{G}}_{n-1})^{-S_{n-1}}.$$
 Again, $S_n$-anti-invariance forces a section $\omega\in\text{Ker}(\hat{R}_n)$ to vanish on all $\hat{D}_{ij}$. But then we follow the argument used in Lemma \ref{pullback lands in BG} to deduce $\text{Res}_{\hat{D}_I}(\omega)=0$ for all $I\subset[n]$. Indeed, $\omega_I:=\text{Res}_{\hat{D}_I}(\omega)$ is a form on $\hat{D}_I\simeq \hat{X}^{[n]/I}\times\hat{\mathbf{P}}^I$ which satisfies $\text{Res}_{T_\sigma}(\omega_I)=0$ for all binary trees $T_\sigma, \sigma\in S_{I}$, because of the commutativity of residue along normal crossing divisor property. Then Diagram \ref{compatibility diagram}, together with injectivity of $\Psi_{T_\sigma}^*,\Psi_I^*$, forces $\omega_I=0$.

We summarize the preceding discussions with the following commutative diagram:
$$
\begin{tikzcd}
0 \arrow[r] & \Gamma(\hat{X}^{n},{p^*(S^n(\mathcal{G}_1))}) \arrow[r]         & \Gamma(\hat{X}^n,\hat{\mathcal{G}}_n)^{-S_n} \arrow[r, "\hat{R}_n"]                   &\Gamma(\hat{X}^{n-1}, \hat{\mathcal{G}}_{n-1})^{-S_{n-1}} \arrow[r]         & 0 \\
0 \arrow[r] & \Gamma(\hat{X}^n,{p^*(S^n(\mathcal{G}_1))}) \arrow[u] \arrow[r] & \Gamma(\hat{X}^n,p^*(\mathcal{G}_n))^{-S_n} \arrow[u] \arrow[r, "R_n"] & \Gamma(\hat{X}^{n-1},{p^*\mathcal{G}_{n-1})^{-S_n}} \arrow[u] \arrow[r] & 0
\end{tikzcd}
$$
The theorem tautologically holds for $n=1,2$. Thus we are finished by induction and the 5 lemma. Note, Theorem \ref{differential operators and forms on Sigma} implies both above exact sequences, on global sections, are 
\begin{equation*}0\rightarrow S^n(J_P^{1,PD}(\mathcal{M}))\rightarrow J_P^{n,PD}(\mathcal{M})\rightarrow J_P^{n-1,PD}(\mathcal{M})\rightarrow 0.\qedhere\end{equation*}
\end{proof}

\subsection{Jet spaces and the Lie cooperad}\label{Jet spaces and the Lie cooperad}
There is a close relationship between the sheaves $\hat{\mathcal{G}}_n$ and the Lie operad $\mathcal{L}ie$ which we now explain. We will only provide a brief exposition to the theory of operads -- for a careful treatment, we refer the reader to e.g \cite{LV}. 

\subsubsection{Operads}
Let $k$ be an algebraically closed field, and let $\mathcal{P}$ be a k-linear operad. Let $V$ be a $k$-vector space. The \textit{free $\mathcal{P}$-algebra generated by $V$}, resp. \textit{free $\mathcal{P}$-algebra with divided powers generated by $V$}, defined by \cite{Fre}, is:
\begin{align}
\mathcal{P}\langle V\rangle&:=\bigoplus_{n\geq 1}(\mathcal{P}(n)\otimes_{k}V^{\otimes n})_{S_n},\;\text{ resp.}\\
\mathcal{P}^{PD}\langle V\rangle&:=\bigoplus_{n\geq 1}(\mathcal{P}(n)\otimes_{k}V^{\otimes n})^{S_n}
\end{align}
The difference between these algebras is the former involves $S_n$-coinvariants while the latter involves $S_n$-invariants. 

We in fact work with cooperads, so let us recall definitions in this setting:
\noindent A \textit{comodule over a $\mathcal{P}$-coalgebra $A$} is a $k$-vector space $M$ together with a collection of linear maps
$$q_n:M\rightarrow M\otimes A^{\otimes (n-1)}\otimes \mathcal{P}(n)$$
satisfying coassociativity, equivariance, and counit axioms. A \textit{comodule $\mathcal{M}$ over an cooperad $\mathcal{P}$} (\cite{KM}), is a collection of vector spaces $\mathcal{M}(n)$ with $S_n$-actions together with linear maps 
$$p_{\alpha_1,\dots,\alpha_k}:\mathcal{M}(\alpha_1+\dots+\alpha_k)\rightarrow \mathcal{M}(k)\otimes\mathcal{P}(\alpha_1)\otimes\dots\otimes\mathcal{P}(\alpha_k)$$
satisfying coassociativity, equivariance, and counit axioms.

\subsubsection{Jet spaces and Lie operad}

Let us recall the $\mathcal{L}ie$ operad. Let $x_1,\dots, x_n$ be variables, with the natural $S_n$-action. Consider the $(n-1)!$-dimension vector space
$$\mathcal{L}ie(n):=\text{Span}\langle [x_{s(1)},[\dots,[x_{s(n-1)},x_n]\dots]]:s\in S_{n-1}\rangle.$$
This forms an operad via the insertion of the bracket compositions. Next, define $$\Omega_{\hat{\mathbf{P}}}:=\{\Omega(\hat{\mathbf{P}}^n,\mathring{\mathbf{P}}^n)\otimes \wedge^n(k^n)\}_{n\geq 1}.$$
Under the perfect pairing in Proposition \ref{Lie and log}, this forms a cooperad, which we will call \textit{the Lie cooperad}. It is instructive to explain the composition maps explicitly.

Recall the stratification on $\hat{\mathbf{P}}^n$: The codimension 1 strata are normal crossing, labeled by $\tilde{D}_{I}$, for $I\subset[n]$ a proper subset of size at least 2. And the closure of the codimension $r$ strata are 
\begin{equation}\label{strata of hat{P}}
\tilde{S}_{I_1,\dots,I_r}:=\tilde{D}_{I_1}\cap\tilde{D}_{I_1\circ I_2}\cap\dots\cap\tilde{D}_{I_1\circ\dots\circ I_r}\simeq \hat{\mathbf{P}}^{I_1}\times\hat{\mathbf{P}}^{I_2}\times\dots\times \hat{\mathbf{P}}^{I_r}\times\hat{\mathbf{P}}^{n/(I_1\circ\dots\circ I_r)}
\end{equation}
This shows $\hat{\mathbf{P}}:=\{\hat{\mathbf{P}}^n\}_{n\geq 1}$ is a topological operad. Thus the composition maps of $\Omega_{\hat{\mathbf{P}}}$ are
$$\text{Res}_{T_{I_1,\dots,I_r}}:\Omega_{\hat{\mathbf{P}}}(\alpha_1+\dots+\alpha_k)\rightarrow\Omega_{\hat{\mathbf{P}}}(k)\otimes\bigotimes_{i=1}^k\Omega_{\hat{\mathbf{P}}}(\alpha_i).$$
The coassociativity axiom is satisfied because the divisors are normal crossing, so the order in which the residue is computed is the same (up to a sign, which is accounted for by the determinant).

Next, for $X$ a smooth projective curve over $k$, the collection of the resolution of the diagonals, $\hat{X}:=\{\hat{X}^n\}_{n\geq 1}$, forms a module over the topological operad $\hat{\mathbf{P}}$ because the strata $S_T$ of $\hat{X}^n$ satisfy \ref{strata decomposition}. Moreover,
$$\Omega_{\hat{X}}:=\{\Omega(\hat{X}^n,\mathring{X}^n)\otimes\wedge^n(k^n)\}_{n\geq 1}$$
forms a comodule over the cooperad $\Omega_{\hat{\mathbf{P}}}$ via the composition maps
$$\text{Res}_{T_{I_1,\dots,I_r}}:\Omega_{\hat{\mathbf{X}}}(\alpha_1+\dots+\alpha_k)\rightarrow\Omega_{\hat{\mathbf{X}}}(k)\otimes\bigotimes_{i=1}^k\Omega_{\hat{\mathbf{P}}}(\alpha_i).$$

Next, we sheafify. Taking global sections of the $\hat{\mathbf{P}}^I$-contribution in Equation \ref{Phi_I^* step 2} produces a morphism of $\mathcal{O}_X$-modules:
$$\Phi_I^*:\mathfrak{g}_P^*\rightarrow (\mathfrak{g}_P^*)^{\otimes I}\otimes \Omega({\hat{\mathbf{P}}^I,\mathring{\mathbf{P}}^I}).$$
This gives $\mathfrak{g}_P^*$ the structure of a sheaf of coalgebras over the Lie-cooperad. Now, consider:
$$\Omega_{\hat{X}}^{PD}\langle{\mathfrak{g}}_P^*\rangle:=\bigoplus_{n\geq 1}\bigg(\Gamma(\hat{X}^n,(\hat{\mathfrak{g}}_P^*)^{\boxtimes n}\otimes\Omega_{\hat{X}^n,\mathring{X}^n}\otimes\omega_{\hat{X}^n})\bigg)^{S_n}$$
where $\omega_{\hat{X}^n}$ is the determinant line bundle on $\hat{X}^n$, and $(\hat{\mathfrak{g}}_P^*)^{\boxtimes n}$ is the pullback of $(\mathfrak{g}_P^*)^{\boxtimes n}$ under $p:\hat{X}^n\rightarrow X^n$, respectively. The comodule morphisms (Recall diagram \ref{diagonal diagram} for the notation):
\begin{equation}\label{Residue as cooperad}
\text{Res}_{\hat{D}_{\alpha}}:\bigg((\hat{\mathfrak{g}}_P^*)^{\boxtimes m}\otimes\Omega_{\hat{X}^m,\mathring{X}^m}\bigg)\vert_{\hat{D}_\alpha}\rightarrow \bigg((\hat{\mathfrak{g}}_P^*)^{\boxtimes [m]/\alpha}\otimes\Omega_{\hat{X}^{[m]/\alpha},\mathring{X}^{[m]/\alpha}}\otimes\tilde{f}_1(\mathfrak{g}_P^*)^{\otimes(\alpha-1)}\bigg)\boxtimes \Omega_{\hat{\mathbf{P}}^\alpha,\mathring{\mathbf{P}}^{\alpha}}.
\end{equation}
endow $\Omega_{\hat{X}^n}^{PD}\langle \mathfrak{g}_P^*\rangle$ with the structure of a comodule over $\mathfrak{g}_P^*$ over the Lie cooperad. If we take global sections over the $\hat{\mathbf{P}}^{\alpha}$ contribution and then take a fiber over $\hat{X}^{[n]/\alpha}$, this morphism just says the free divided power space generated by $\mathfrak{g}^*$,
$$\Omega_{\hat{\mathbf{P}}}^{PD}\langle\mathfrak{g}^*\rangle:=\bigoplus_{n\geq1}\big((\mathfrak{g}^*)^{\otimes n}\otimes \Omega(\hat{\mathbf{P}}^{[n]},\mathring{\mathbf{P}}^{[n]})\big)^{S_n},$$
is a comodule over $\mathfrak{g}^*$ over the Lie cooperad. 
These observations are what allow us to make the final claim, and in particular suggest why $\Gamma(\hat{X}^n,\mathcal{G}_n)^{-S_n}$ has divided powers.

\begin{cor}\label{jets and Lie operad} The collection of BG-sheaves $\{\hat{\mathcal{G}}_n\}_{n\geq 1}$ form a sheaf over $\hat{X}:=\{\hat{X}^n\}_{n\geq 1}$, in the sense of \cite[Section 1.5]{GK}. Here, $\hat{X}$ is viewed as a topological module over the operad $\hat{\mathbf{P}}$. Moreover, 
$$\bigoplus_{n\geq 1}\Gamma(\hat{X}^{n},\hat{\mathcal{G}}_n\otimes\omega_{\hat{X}^n})^{S_n}$$
is a sub $\mathfrak{g}_P^*$-comodule of $\Omega_{\hat{X}}^{PD}\langle{\mathfrak{g}}_P^*\rangle$. Here, $\mathfrak{g}_P^*$ is viewed as a sheaf of coalgebras over the Lie cooperad.
\end{cor}
\begin{proof}
Both of the claims follow by observing that restricting the residue map in Equation \ref{Residue as cooperad} to $\hat{\mathcal{G}}_n$ induces the map of $\mathcal{O}_{\hat{D}_\alpha}$-modules (recall diagram \ref{diagonal diagram} for the notation):
\begin{equation}\label{Residue on jets}
\text{Res}_{\hat{D}_{\alpha}}:\hat{\mathcal{G}}_m\vert_{\hat{D}_\alpha}\rightarrow \hat{\mathcal{G}}_{[m]/\alpha}\otimes\tilde{f}_1^*(\mathfrak{g}_P^*)^{\otimes (\alpha-1)}\boxtimes \Omega_{\hat{\mathbf{P}}^\alpha,\mathring{\mathbf{P}}^{\alpha}}.
\end{equation}
 Indeed, checking this boils down to a similar proof as given in Lemma \ref{Residue as map of jets}. The main point is there is a compatibility \ref{compatibility}
\begin{equation*}
\text{Res}_{T}\circ(\Psi_{I}^*\vert_{S_T}) = \Psi_T^*\;\;\;\text{ for all $I$-trees $T$.}\qedhere \end{equation*}
\end{proof}


\begin{thebibliography}{99}
 \bibitem[BD]{BD} A. Beilinson, V. Drinfeld,  {Affine Kac-Moody algebras and polydifferntials},  {\em International Math. Research
Notices}, (1994), 1--11.
 \bibitem[BG]{BG} A. Beilinson, V. Ginzburg, {Infinitesimal structure of moduli space of $G$-bundles,}  {\em International Math. Research
Notices}, (1992), 63--74.
 \bibitem[BMR]{BMR} R. Bezrukavnikov, I. Mirkovic, D. Rumynin, {Localization of modules for a semisimple Lie algebra in prime characteristic} {\em Annals of Mathematics} (2008), 945--991.
 \bibitem[BO]{BO} P. Berthelot, A. Ogus, {Notes on crystalline cohomology}, {\em Princeton university press}, Vol. 21, (2015).
 \bibitem[BZF]{BZF} D. Ben-Zvi, E. Frenkel, Vertex algebras and algebraic curves, {\em American Mathematical Soc.}, No. 88, (2004).
\bibitem[EV92]{EV92} H. Esnault, E. Viehweg, Lectures on vanishing theorems. {\em Basel: Birkhauser}, Vol. 20. (1992).
\bibitem[EV94]{EV94} H. Esnault, E. Viehweg, Higher Kodaira-Spencer classes, {\em Mathematische Annalen}  \textbf{299.3} (1994), 491--528.

\bibitem[Fre]{Fre} B. Fresse. On the homotopy of simplicial algebras over an operad. {\em Trans. Amer. Math. Soc., 352} \textbf{9} :4113--4141, (2000).
\bibitem[FM94]{FM94} W. Fulton, R. MacPherson, A compactification of configuration spaces, {\em Annals of Mathematics} \textbf{139.1} (1994), 183--225.
\bibitem[G]{G} V. Ginzburg, Resolution of the diagonals, {\em The Moduli Space of Curves}  Boston, MA: Birkhauser Boston, (1995)  231--266.
\bibitem[GK]{GK} V. Ginzburg, M. Kapranov, Koszul duality for operads, {\em Duke Math. Journal}. \textbf{76} (1994),  203--272.
\bibitem[KM]{KM} M. Kapranov and Y. Manin, Modules and Morita Theorem for Operads, {\em American Journal of Mathematics} \textbf{123.5} (2001), 811--838. 
\bibitem[LV]{LV} J-L. Loday, B. Vallette, Algebraic operads {\em Grundlehren der mathematischen Wissenschaften}, Springer, Heidelberg, \textbf{346}, (2012).
\bibitem[Ran1]{Ran1} Z. Ran,  Canonical infinitesimal deformations, {\em J. Algebraic Geom} \textbf{9.1} (2000) 43--69.
\bibitem[Ran2]{Ran2} Z. Ran, Derivatives of Moduli, {\em International Math. Research
Notices} (1993), 93--106. 
\bibitem[Sor99]{Sor99} C. Sorger, Lectures on moduli of principal G-bundles over algebraic curves,  School on Algebraic Geometry (Trieste), ICTP Lect. Notes, Abdus Salam International Centre for Theoretical Physics, Trieste, (1999).
\bibitem[S76]{S76} K. Saito, On a generalization of de Rham lemma, {\em Annales de l'institut Fourier}. \textbf{26.2} (1976).
\bibitem[S79]{S79} K. Saito, Theory of logarithmic differential forms and logarithmic vector fields, {\em J. Fac. Sci. Univ. Tokyo Sect.} IA Math \textbf{27.2} (1980), 265--291.
\bibitem[Tate]{Tate} J. Tate, Residues of differentials on curves, {\em Annales scientifiques de l'Ecole Normale Superieure.} \textbf{1.1} (1968).
\bibitem[V82]{V82} E. Viehweg, Vanishing theorems. {\em Journal fur die reine und angewandte Mathematik} \textbf{335} (1982), 1--8. 
\end{thebibliography}
\end{document}